\documentclass[a4paper, 11pt]{scrartcl}
	\usepackage[a4paper,left=3cm,right=3cm,top=3cm,bottom=4cm]{geometry}
	\usepackage{caption}
 		\usepackage{pgfplots} \pgfplotsset{compat=1.12}

	\usepackage[arrow, matrix, curve]{xy} % für geile kommutative Diagramme
	\usepackage{enumitem} % gute Nummerierung
	\usepackage{amsmath}   % enthaelt nuetzliche Makros fuer Mathematik
	\usepackage{amsthm}    % fuer Saetze, Definitionen, Beweise, etc.
	\usepackage{bm}			% Dick in Formeln
	\usepackage{amsfonts}  % spezielle AMS-Mathematik-Fonts
	\usepackage{graphicx}  % erlaubt einfaches Einbinden von Graphiken
	\usepackage{float}	 % erlaubt es (unsch\"on) floats an eine bestimmte Stelle zu zwingen
	\usepackage{mathrsfs} % damit bekomme ich geschwungene Buchstaben
	\usepackage{amssymb} %für semidirekte Produkte
	\usepackage{extarrows} %für beschriftung über und unter pfeilen
	\usepackage{mathtools} %da ist ein definiert als zeichen bei
  \usepackage{thmtools}
  \usepackage{thm-restate}
  \usepackage{hyperref}
  \usepackage{cleveref}
  \usepackage{tocloft}
	
  \newcommand{\N}{\ensuremath{\mathbb{N}}}   % natuerliche Zahlen
  \newcommand{\R}{\ensuremath{\mathbb{R}}}   % reelle Zahlen
  \renewcommand{\epsilon}{\varepsilon}       % damit epsilon und phi so
  \renewcommand{\phi}{\varphi}		         % aussehen wie gewohnt	
  \newtheorem{satz}{Satz}[section]
  \newtheorem{lemma}[satz]{Lemma}

  \newtheorem{theorem}[satz]{Theorem}
  \newtheorem{remark}[satz]{Remark}
  \theoremstyle{definition}
  \newtheorem{definition}[satz]{Definition}
  
  \theoremstyle{remark}
  \makeatletter                    % Fussnote ohne Symbol
   \def\blfootnote{\xdef\@thefnmark{}\@footnotetext}
\begin{document}

\title{On Fokker-Planck Equations with \\ In- and Outflow of Mass}
\author{Martin Burger\thanks{Friedrich-Alexander-Universit\"at Erlangen-N\"urnberg, Germany (martin.burger@fau.de).} \ \
Ina Humpert\thanks{Applied Mathematics M\"unster: Institute for Analysis and Computational Mathematics,
Westf\"alische Wilhelms-Universit\"at (WWU) M\"unster, Germany (ina.humpert@uni-muenster.de).} \ \
Jan-Frederik Pietschmann\thanks{Technische Universit\"at Chemnitz, Fakult\"at f\"ur Mathematik, Germany (jfpietschmann@math.tu-chemnitz.de).} \ \
}
\maketitle
\begin{abstract}
% \textbf{Abstract}
Motivated by modeling transport processes in the growth of neurons, we present results on (nonlinear) Fokker-Planck equations where the total mass is not conserved. This is either due to in- and outflow boundary conditions or to spatially distributed reaction terms. We are able to prove exponential decay towards equilibrium using entropy methods in several situations. As there is no conservation of mass it is difficult to exploit the gradient flow structure of the differential operator which renders the analysis more challenging. In particular, classical logarithmic Sobolev inequalities are not applicable any more. Our analytic results are illustrated by extensive numerical studies.

\vspace{2ex}

\end{abstract} 

\section{Introduction}
A lot of recent research effort was devoted to gain understanding of the develoment of neuron cells, which are highly relevant for brain function and disfunction, but still their growth process is far from being completely understood.
Takano et al. stated that neurons develop "structurally and functionally distinct processes called axons and dendrites" \cite[p. 1]{TakanoetatNeuronalPolarization} and while in the beginning all neurons look the same, over time they polarize and generate only one axon but multiple dendrites, see Figure \ref{figureneuron}.
Currently, most mathematical models deal with the influence of proteins in this process, modeled by systems of reaction diffusion equations. Here, we are interested in a different aspect, namly in the transport of vesicles from the cell nucleus to the neurite tips. Our model is based on the Fokker-Planck equation 
% To shed some light on this poorly understood process called \textit{axonification} we are looking for appropriate models based on the \textit{Fokker-Planck equation} 
\begin{align}\label{eq:fokker}
	\partial_t \rho + \nabla \cdot (-\nabla \rho + f(\rho) \nabla V) = 0,
\end{align}
% describing particle movement 
where $\rho=\rho(x,t)$ denotes the density of vesicles and $V=V(x)$ is a given potential. The function $f(\rho)$ is chosen to be either $f(\rho)=\rho$ or $f(\rho)=\rho(1-\rho)$. While the first case is just linear transport, the second choice enforces the bound $\rho \le 1$ on the vesicle density. To model the in- and outflux of vesicles, we will supplement \eqref{eq:fokker} either with flux boundary conditions or reaction terms. For the first approach we divide the boundary of the domain $\Omega$ into three parts: An inflow region, an outflow region and an insulated part (cf. \cite{BurgerPietschmannFlowCharacteristics}). On the inflow part, we prescribe a fixed flux while on the outflow region, the flux is proportional to the density. For the second model, we add no-flux boundary conditions and reaction terms for the in- and outflow of vesicles. This can be seen as an averaging of in- and outflow boundary conditions in a thin higher-dimensional structure. 

In all cases, the mass of $\rho$ changes during its evolution making it difficult to exploit the formal gradient flow structure of the equations (cf.  \cite{jordan1998variational,Otto2001}),
\begin{align}\label{eq:fokkergf}
	\partial_t \rho =\nabla \cdot (f(\rho) \nabla E'(\rho)) = 0, \qquad E(\rho) = \int_\Omega F(\rho) - \rho V~dx,
\end{align}
where $F'(\rho) = \frac{1}{f(\rho)}. $  We will partly exploit the underlying gradient flow structure by considering a relative entropy (or Bregman distance) to a stationary solution $\rho_\infty$ instead, namely
$$ E(\rho|\rho_\infty) = \int_\Omega F(\rho) - F(\rho_\infty) - F'(\rho_\infty)(\rho - \rho_\infty)~dx. $$
After verifying existence and uniqueness of a stationary solutions, we can use the dissipation of the relative entropy 
to show that solutions of the PDE decay exponentially to the respective stationary or equilibrium solution (we call a stationary solution an equilibrium if the flux vanishes, i.e. $f(\rho_\infty) \nabla E'(\rho_\infty) = 0$). 

% 
% To generate our model we started with a size exclusion model in a periodic lattice where one single species can either jump to its neighboring site at the left or at the right with equal probability and generated the continuum limit of this system, see e.g. \cite{BurgerFrancescoPietSchlakeNonlinearCrossdiffusion}.
% Next we model the additional flux by using two different methods.

% In this case, we also consider case of a nonlinear Fokker-Planck equations with density constraint.
%  
% Whereas the second method is to simply add two terms - one corresponding to the influx and one to the outflux - to the PDE.
% As the second method turns out to be easier to understand we also expand the model that it is capable to give the upper bound 1 for the particle density allowing us to call $\rho(x,t)$ the density of a particle.
% 
% A comparable model was derived in \cite{WoodATotallyAsymmetriExclusionProcessWithStochasticallyMediatedEntranceandExit}, where stochastic entrance and exit conditions were modeled in \textcolor{red}{1D CHECK: MACHT DER DEN LIMES?}
% and one can also investigate similar models where two types of species are included, see e.g. \cite{berendsen2017cross}.
% 
%For all models we use entropy methods to show that solutions of the PDE decay exponentially to the respective stationary or equilibrium solution. 
% We can characterize the rates in terms of functional inequalities, as done in \cite{fellner2017uniform} and \cite{DesvillettesFellnerExponentialdecay} among others. 
Entropy methods are very convenient to analyse the long-time behaviour of linear and non-linear partial equations and are strongly related to functional inequalities like the logarithmic-Sobolev inequality, see \cite{MarkowichVillaniOntheTrend}. In case of in- and outflow terms in the bulk we can directly exploit the dissipation generated by the reaction terms in order to show exponential convergence to equilibrium, similar to recent approaches for reaction-diffusion equations (cf. \cite{mielke2011gradient,liero2013gradient,DesvillettesFellnerExponentialdecay,fellner2017uniform,HaskovecHittmeirMarkowichMielkeDecay}). In the case of in- and outflow boundary condition, exponential convergence needs to be shown using the bulk dissipation by the diffusion and transport. However, standard logarithmic Sobolev inequalities do not apply in our case as the total mass is not conserved. In particular, using a scaling argument, we can give a counter example in our setting. This is in contrast to many similar models in the literature, (see   \cite{MarkowichVillaniOntheTrend,ArnoldCarilloEntropiesandEquilibria}). However, in the case when the differential operator is linear, we can resort to a variation of Friedrichs' inequality taking the boundary values into account. This allows us to bound the entropy dissipation in terms of the relative entropy by which, together with a Gronwall-type argument, we recover exponential decay of the relative entropy. Combining this with a Csisz\'{a}r-Kullback inequality, one also obtains decay in the $L^1$-norm. 
%For a detailed review on entropy methods see \cite{ArnoldCarilloEntropiesandEquilibria}.

This paper is organized as follows: In section \ref{sectionBiologicalBackground} we explain the biological background of the models. In section \ref{sectionBoundaryInandOut}, we present the linear model where in- and outflow terms are modeled by boundary in- and outflow. In section \ref{sectionUniformSpatialInandOut} we investigate a model with spatially distributed in- and outflow and in section \ref{sectionUniformStatialInandOutConstraint} we combine this model with a density constraint. 

\section{Biological Background and Modelling}
\label{sectionBiologicalBackground}
Neurons are the major part of the central nervous system receiving and transmitting information through the human body. A typical neuron consists of a cell nucleus and two types of `arms' originating in the cell body (see Figure \ref{figureneuron}).
\begin{figure}[t]
	\centering
    \includegraphics[width=0.4\textwidth]{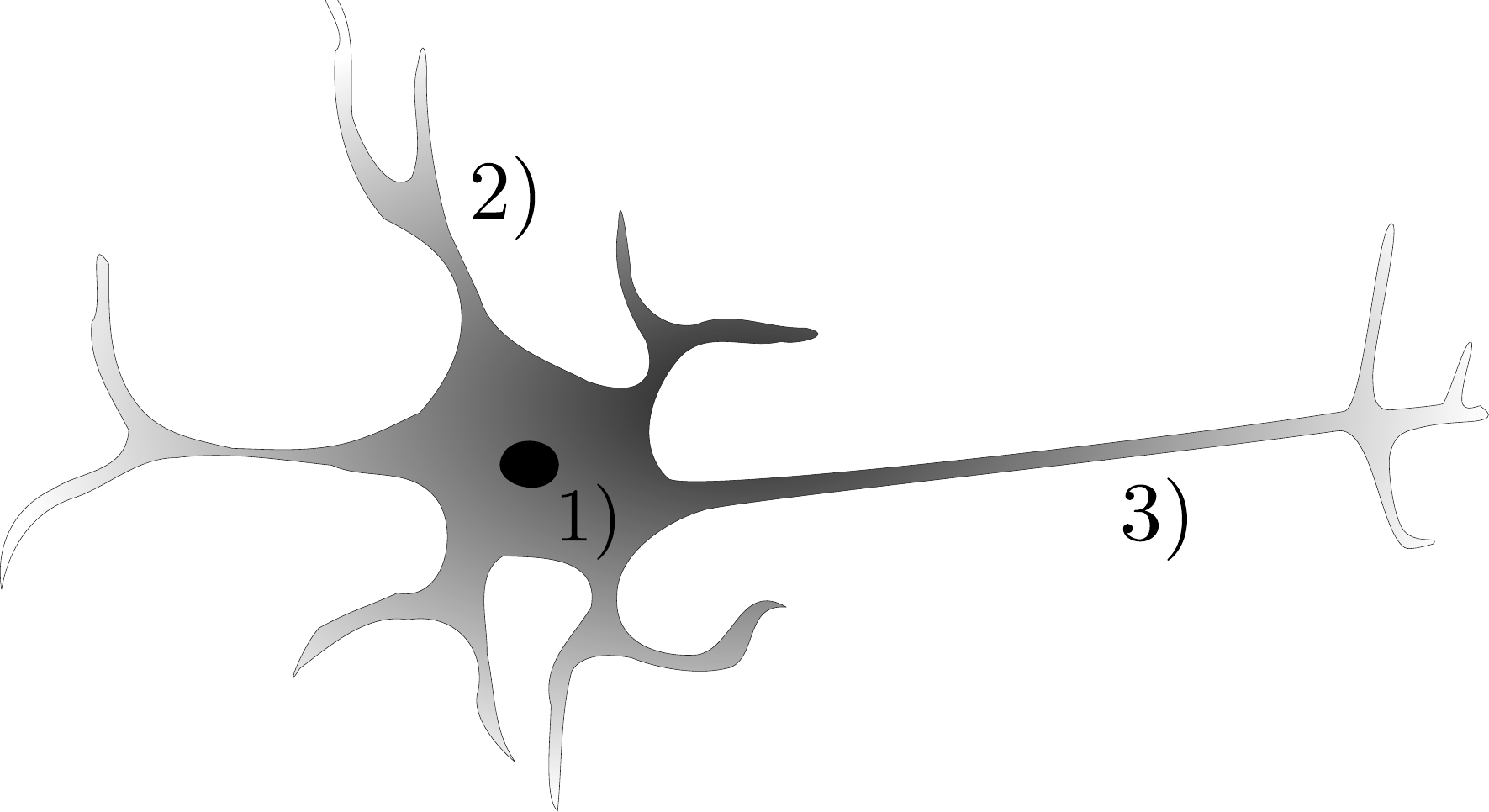}
	\caption{Sketch of a neuron (1) cell nucleus, 2) dendrite and 3) axon)}
	\label{figureneuron}
\end{figure}
These arms are called dendrites and axons. 
Each neuron has several dendrites but only one axon.
``An axon is typically a single long process that transmits signals to other neurons by the release of neurotransmitters. 
Dendrites are composed of multiple branches processes and dendritic spines, which contain neurotransmitter receptors to receive signals from other neurons" \cite[p. 1]{TakanoetatNeuronalPolarization}.
The formation of these different processes, called \textit{polarization}, is crucial for a proper functionality of the central nervous system.

The typical polarization of a neuron in vitro is divided into five stages.
In the first stage the neuron extends filopodias around the cell body,  which are ``thin, actin-rich plasma-membrane protrusions that function as antennae for cells to probe their environment" \cite[p. 1]{MattilaFilopdia}. 
In stage two these filopodia develop into neurites which in the beginning are all equivalent and seem to grow and shrink randomly. 
The actual polarization starts in stage three where one minor neurite grows quicker than the others and develops into the future axon.
In stage four the remaining neurites shrink into dendrites and in stage 5 the polarized neuron matures (see also the poster on neuronal polarization in \cite{TakanoetatNeuronalPolarization}).

For years biologists have been trying to understand the molecular machinery hidden behind this procedure developing many explanation approaches, see \cite{NaoyukietalSymmetrybreaking}, which are mainly based on the different concentration of certain proteins in the neurites see \cite{NambaetalExtracellularSignaling}. It is conjectured that the neurite growth is mainly driven through a certain vesicle flux. These vesicles are believed to merge with the cell membrane at the neurite tips making the neurite grow.
On the other hand it is conjectured that a part of the membrane can also be separated forming a vesicle and making the neuron shrink.
The flux of theses vesicles can be measured by special microscopes.

To better understand the growth of the neurite in stage two which seems randomly we suggest a mathematical model describing the vesicle flux in the axons. In the most general case we consider the three-dimensional neurite with several types of different boundary conditions. Thus, in the  three-dimensional domain $\Omega_N$ modelling the neurite the vesicle density satisfies a Fokker-Planck equation
\begin{equation}
\partial_t \rho = \nabla \cdot J, \qquad J = - \nabla \rho - f(\rho) u, 
\end{equation}
with a three-dimensional velocity field $u$. This is complemented by a boundary condition on the flux $\partial \Omega_N$ (with $n$ denoting the outward unit normal):
\begin{equation}
J \cdot n = -  a \frac{f(\rho)}\rho + b \rho, 
\end{equation}
where the nonnegative functions $a$ and $b$ model rates of in- and outflow, respectively. Typically the supports of $a$ and $b$ do not intersect and we find up to three different regions on the boundary, namely the inflow part where $a$ is positive, the outflow part where $b$ is positive, and the isolated part where $a=b=0$. The influx term in our model therefore corresponds to vesicles entering the neurite, which happens mainly at the part of the boundary that is an interface to the cell nucleus, while outflux corresponds to vesicles merging with the cell membrane, which typically happens at the neurite tip. Note that in its present form, our model does not yet explicitly account for the growth of the neurite, this might however be encoded in the transport terms when rescaling the domain.  Moreover, since frequently the directed transport along microtubuli dominates over intracellular fluid transport in neurites, it seems reasonable to assume a potential force $u= - \nabla V$. 

If we consider the neurite as an almost axisymmetric structure with small diameter, i.e., 
$$ \Omega_3 = \{ (x_1,x_2,x_3) \; : \; x_1 \in (0,1),\; (x_2,x_3) \in \Omega_2(x_1) \}, $$ %\cup_{x_1 \in (0,1)} \Omega_2(x_1)$$ 
for some $\Omega_2(x_1) \subset \R^2$ with diameter much smaller than one, we can make further approximations. In particular the equilibration orthogonal to the axis, which we assume to be the $x_1$ direction, will be fast, hence 
$$ \rho(x,t) \approx \rho_1(x_1,t) q(x,t), $$
where $q$ is a stationary solution of 
$$\nabla_{23} \cdot ( \nabla_{23} q + \frac{f(\rho_1 q)}{\rho_1} \nabla_{23} V) = 0, $$
where $\nabla_{23}$ denotes the gradient with respect to $(x_2,x_3)$. On the boundary, $q$ satisfies
\begin{equation}
 -  ( \nabla_{23} q + \frac{f(\rho_1 q)}{\rho_1} \nabla_{23} V) \cdot n = -  a ~ \frac{f(\rho_1 q)} {\rho_ 1 q} + b \rho_1 q. 
\end{equation}
Now, taking an average orthogonal to the axis in the small cross-section $\Omega_2(x_1)$, we find with Gauss' theorem
\begin{equation}  \partial_t (\rho \overline{q}) =  
\partial_{x_1}  ( \overline{q} (\partial_{x_1} \rho_1 + g(\rho_1) \partial_{x_1} V) )
+  A(\rho_1)   -  \beta \rho_1
\end{equation}
with
\begin{align*}
 \overline{q}(x_1,t) &= \int_{\Omega_2(x_1)} q(x,t) ~d(x_2,x_3),  \\
  g(\rho_1(x_1,t)) &=  \frac{1}{ \overline{q}(x_1,t) } \int_{\Omega_2(x_1)} f(\rho_1(x_1,t) q(x,t)) ~d(x_2,x_3), \\
	A(x_1,\rho(x_1,t),t) &=  \int_{\partial \Omega_2(x_1)} a(x)~ \frac{f(\rho_1 q)} {\rho_ 1 q} ~ds, \\
	\beta(x_1,t) &=  \int_{\partial \Omega_2(x_1)} b(x) q(x,t) ~ds .
\end{align*}
Hence, the in- and outflow boundaries naturally lead to analogous reaction terms in the bulk, which motivates the study of such models as well. We mention that we can obtain a two-dimensional version of the equations in geometries approximating a thin sheet as well. 

Let us mention that the models simplify in the special cases of functions $f$ we consider. In the linear case, with $f$ being the identity, we find
$$ g(\rho_1) = \rho_1, \qquad A(x_1,\rho_1,t) = \int_{\partial \Omega_2(x_1)} a(x)~ds = \alpha(x_1). $$ 
Thus, the resulting equation is simply
\begin{equation}  \partial_t (\rho \overline{q}) =  
\partial_{x_1}  ( \overline{q} (\partial_{x_1} \rho_1 + \rho_1 \partial_{x_1} V) )
+  \alpha  -  \beta \rho_1.
\end{equation}
In the crowded case, we still have
$$  A(x_1,\rho(x_1,t),t) =  \int_{\partial \Omega_2(x_1)} a(x)(1- \rho_1 q) ~ds = (a_0 - a_1\rho_1). $$ 
Hence, in both cases, the reaction terms have the same shape as the boundary terms and it is consequently natural to study the following cases:
\begin{itemize}
\item The linear Fokker-Planck equation $f(\rho)=\rho$ with in- and outflow boundary conditions, which will be the subject of section 3.

\item The crowded Fokker-Planck equation  $f(\rho)=\rho(1-\rho)$ with in- and outflow boundary conditions, which was done in \cite{BurgerPietschmannFlowCharacteristics}.

\item The linear Fokker-Planck equation  $f(\rho)=\rho$ with reaction terms of the form $a- b \rho$, which is the subject of section 4.

\item The crowded Fokker-Planck equation  $f(\rho)=\rho(1-\rho)$ with reaction terms of the form $a (1-\rho) - b \rho$, which is the subject of section 5.

\end{itemize}

%To this end we model a neurite as an one-dimensional object which we identify with the unit interval, where the point $x=0$ corresponds to the beginning of the neurite i.e. the nearest point of the neurite to the cell nucleus, and $x=1$ corresponds to the neurite tip.

% \textcolor{red}{If there is some equilibrium state of vesicles after some time, the neuron should grow not randomly anymore but following a clear systematic of growth.}

\section{Linear Model with Boundary In- and Outflux}
\label{sectionBoundaryInandOut}

% We are first having a look at the Fokker-Planck equation motivated by the conservation law coming from physics.
% This law states that during time evolution a certain measurable parameter of an isolated physical system does not change, so we write
We start by considering the linear Fokker-Planck equation
\begin{align}
	\label{Fokker-Planck1A}
	 \partial_t \rho + \nabla \cdot J=0\;
	\text{  with  } \;
	J =- \nabla \rho + \rho \nabla V,\quad \text{ on } \Omega \times (0,T)
\end{align}
for given $T\geq 0$, $x \in \Omega \subset \mathbb{R}^n$, and where $V = V(x)$ is a given potential. The unknown function $\rho=\rho (x,t) $ describes the density of vesicles and we supplement the equation by the flux boundary conditions 
% an ensemble of particles in some domain  and should fulfill the boundary conditions
\begin{alignat}{3}
	\label{BoundaryCondition1Ain}
	-J \cdot n &= \alpha  \ \ \ &&\text{ on } \Gamma_{\text{in}} &&\times (0,T),
	\\
	\label{BoundaryCondition1Aout}
	J \cdot n &= \beta\rho  \ \ &&\text{ on } \Gamma_{\text{out}} &&\times (0,T), 
	\\
	\label{BoundaryCondition1Aisolated}
	J \cdot n &= 0 \ \ \ \ &&\text{ on } \partial\Omega\setminus\lbrace \Gamma_{\text{in}} \cup \Gamma_{\text{out}} \rbrace &&\times (0,T) ,
\end{alignat}
where $n$ denotes the outward normal. We make the following assumptions: 
\begin{itemize}
	\item[(A1)] The connected and bounded domain $\Omega \subset \mathbb{R}^n$ has Lipschitz boundary $\partial \Omega$.
	\item[(A2)] The potential satisfies $V(x) \in W^{1, \infty}(\Omega)$. %is smooth and bounded.
	\item[(A3)] The initial concentration $\rho_0$ is non negative and fulfills $\rho_0 \in L^2(\Omega)$.
	\item[(A4)] The subsets $\Gamma_{\text{in}}, \Gamma_{\text{out}} \subset \partial\Omega$ of the boundary are open, disjoint and $\Gamma_{\text{out}}$ is nonempty. 
	\item[(A5)] The functions $\alpha$ and $\beta$ satisfy $\alpha \in L^\infty(\Gamma_{\text{in}})$ with $\alpha \geq \alpha_0 > 0$ and $\beta\in L^\infty(\Gamma_{\text{out}})$ with $\beta \geq \beta_0 > 0$ for some $\alpha_0,\beta_0 > 0$.
\end{itemize}
In this setting vesicles enter the domain at $\Gamma_{\text{in}}$ with the rate $\alpha$ and leave it at $\Gamma_{\text{out}}$ with rate $\beta \rho$. An additional insulated part of the domain $\partial\Omega\setminus\lbrace \Gamma_{\text{in}} \cup \Gamma_{\text{out}}\rbrace$ may exist, see also Figure \ref{1ASketchGeometry} for a sketch of such a geometry in 2D. 
Note that in \eqref{BoundaryCondition1Aout} the term $\beta\rho =0$ is zero if and only if $\rho=0$, which means that as soon as vesicles reach the exit, they can leave the domain with rate $\beta \rho$. In particular due to the boundary conditions \eqref{BoundaryCondition1Ain} - \eqref{BoundaryCondition1Aisolated} there is no mass conservation, i.e. the spatial integral over $\rho(x,t)$ changes in time.
\begin{figure}[h]
	\centering 
    \includegraphics[width=0.25\textwidth]{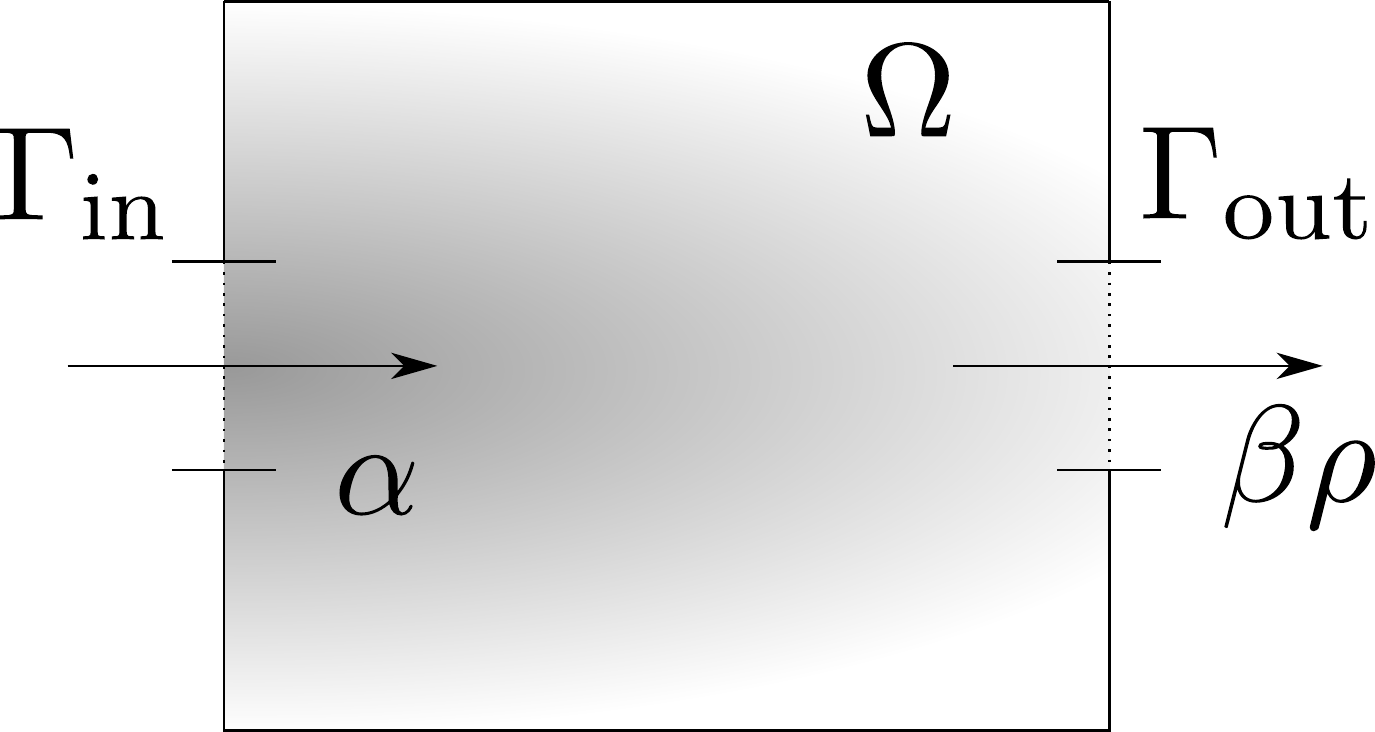}
	\caption{Sketch of the geometry of boundary in- and outflux in 2D with a possible density}
	\label{1ASketchGeometry}
\end{figure}
\begin{remark}
	Choosing the influx parameter $\alpha=0$, the stationary solution is obviously $\rho_\infty=0$ as particles only leave but never enter the domain.
	As exponential convergence can be shown for the classical Fokker-Planck equation, as for example done in \cite{MarkowichVillaniOntheTrend}, it clearly hold in the ''no influx case``, too.
	But as the quadratic relative entropy, see below, is undefined in this case, we will only consider $\alpha>0$ from now on.
\end{remark}
Exponential convergence of the relative entropy is based on the following version of Friedrichs' inequality:
\begin{theorem} \cite[p. 4]{DuduchavaPoincareInequalities}
	\label{FriedrichsInequality}
	For a bounded domain $\Omega$ and functions $u\in H^1(\Omega)$ the \textit{Friedrichs' inequality with boundary values} holds, i.e. for any non empty subset $\Gamma \subseteq \partial \Omega$ of the boundary 
	\begin{align*}
		\int_\Omega u^2 ~ dx
		\leq C_F \Big( \int_\Omega \vert \nabla u \vert^2 ~ dx + \int_{\Gamma} u^2 ~ d\sigma \Big)
	\end{align*}
	holds for a constant $C_F>0$, that we call \textit{Friedrich's constant}.
\end{theorem}
More precisely, we consider the quadratic relative entropy
\begin{align*}
		E(\rho\vert \rho_\infty) 
	= \frac{1}{2} \int_\Omega \frac{(\rho-\rho_\infty)^2}{\rho_\infty} ~ dx,
\end{align*}
where $\rho_\infty$ denotes the solution to the stationary version of \eqref{Fokker-Planck1A}--\eqref{BoundaryCondition1Aisolated} (this will be made precise below). The main result of this section is the following theorem: 
% \ref{SubsectionQuadraticEntropy}:
% 
\begin{theorem} 
	\label{1AEexpDecay}
	Let (A1)-(A5) hold, then the solution $\rho$ to equation \eqref{Fokker-Planck1A} with initial datum $\rho_0$ and boundary conditions \eqref{BoundaryCondition1Ain} - \eqref{BoundaryCondition1Aisolated} obeys the following exponential decay towards equilibrium
	\begin{align*}
	\vert\vert \rho - \rho_\infty \vert\vert^2_{L^1(\Omega)} \leq 2 \max \lbrace \rho_\infty \rbrace E(\rho_0\vert \rho_\infty) e^{-C_F K_1 t} \ \ \forall t \ge 0,
	\end{align*}
	where $K_1=\min\lbrace \frac{\beta_0}{2}, 1\rbrace, \, C_F$ being the constant of Friedrich's inequality and $\rho_\infty$ being the stationary solution.
\end{theorem}

\subsection{Existence and Uniqueness for the Time Dependent Problem}
As $\rho=\rho (x,t) $ describes the density of vesicles, one naturally expects $\rho \geq 0$, which is proven below. But first we introduce the notion of weak solution and prove an existence result.

\begin{definition}
	\label{WeakSolution1A}
	We say that a function $\rho \in L^2(0,T;H^1(\Omega))$ with $\partial_t \rho \in L^2(0,T;H^{-1}(\Omega))$ is a \textit{weak solution} to equation \eqref{Fokker-Planck1A} supplemented with the boundary conditions \eqref{BoundaryCondition1Ain}-\eqref{BoundaryCondition1Aisolated} if the identity
	\begin{align*}
		\int_\Omega \partial_t \rho \, \phi ~ dx
		- \int_\Omega J \cdot \nabla \phi ~ dx 
		+ \int_{\Gamma_{\text{out}}} \beta \rho \phi ~ d\sigma
		= \int_{\Gamma_{\text{in}}} \alpha \phi~  d\sigma
	\end{align*}
	holds for all $\phi \in H^1(\Omega)$ and a. e. time $0 \leq t \leq T$.
	We can rewrite the weak formulation in terms of the \textit{Slotboom variable} $u\coloneqq \rho e^{-V}$. As the transformed flux is given as $J= -e^V \nabla u$, we obtain
	\begin{align}
		\label{1AweakformulationSlotboom}
		\int_\Omega e^V \partial_t u \, \phi ~ dx
		+ \int_\Omega e^V\nabla u \cdot \nabla \phi ~ dx 
		+ \int_{\Gamma_{\text{out}}} \beta ue^V \phi ~ d\sigma
		= \int_{\Gamma_{\text{in}}} \alpha \phi ~ d\sigma.
	\end{align}
\end{definition}

\begin{lemma}
	Let (A1) - (A5) hold. Then there exists a unique weak solution to equation \eqref{Fokker-Planck1A} in the sense of definition \ref{WeakSolution1A}.
\end{lemma}
\begin{proof}
	We use the Slotboom formulation of the problem and use the second and third term on the left side of \eqref{1AweakformulationSlotboom} to define the continuous but non-coercive bilinear form
	\begin{align*}
		B[u, \phi]
		= \int_\Omega e^V\nabla u \cdot \nabla \phi ~ dx 
		+ \int_{\Gamma_{\text{out}}} \beta ue^V \phi ~ d\sigma.
	\end{align*}
	This form fulfills a \textit{G\aa rding inequality}, see e.g. \cite[6.2.2 Theorem 2]{EvansPartialDifferentialEquations}, so that for all $u \in H^1(\Omega)$ we have
	\begin{align}\label{eq:garding}
	\begin{split}
		B[u,u]
		&= \int_\Omega e^V \vert\nabla u \vert^2 ~ dx 
		+ \int_{\Gamma_{\text{out}}} \beta e^V \vert u \vert^2 ~ d\sigma
		\\
		&\geq 
		K_2 \int_\Omega \vert\nabla u \vert^2 ~ dx 
		+ K_2 \beta_0 \int_{\Gamma_{\text{out}}} \vert u \vert^2 ~ d\sigma 		
		= \eta \Vert u \Vert^2_{H^1(\Omega)} - \eta \Vert u \Vert^2_{L^2(\Omega)},
		\end{split}
	\end{align}
	where $K_2 = \inf \lbrace e^V \rbrace, ~ \eta = K_2 C_F^{-1} \min\lbrace 1,\beta_0 \rbrace$ and $C_F$ is the constant coming from the Friedrich's inequality with boundary values, see theorem \ref{FriedrichsInequality}.
	Existence of a unique solution $u$ then follows directly form the ideas stated in \cite[7.1.2]{EvansPartialDifferentialEquations} together with the trace theorem \cite[5.5 Theorem 1]{EvansPartialDifferentialEquations}, applied to the right side of \eqref{1AweakformulationSlotboom}.
\end{proof}

\begin{lemma}
	\label{1AWeakSolNonNeg}
	The weak solution $\rho$ to equation \eqref{Fokker-Planck1A} is non-negative, i.e. $\rho \geq 0 $ in a.e. $x \in \Omega$ and every $t\in T$.
\end{lemma}

\begin{proof}
	If we use the Slotboom-formulation of the problem, $u$ is non-negative if and only if $\rho$ is non-negative.	
	Choosing the test function $u^{-}= \min \lbrace u,0\rbrace$, the weak formulation yields
	\begin{align*}
		\partial_t \int_\Omega e^V \frac{\vert u^-\vert^2}{2} ~dx 
		+ \int_\Omega e^V \vert \nabla u^-\vert^2 ~ dx 
		+ \int_{\Gamma_{\text{out}}} \beta e^V \vert u^-\vert^2 ~ d\sigma
		= \int_{\Gamma_{\text{in}}} \alpha u^- ~ d\sigma. 
	\end{align*}
	Omitting the non-positive right hand side as well as the non-negative second and the third term on the left hand side, and integration with respect to time 
% 	the application of Gronwall's lemma (see e.g. \citep[624]{EvansPartialDifferentialEquations}), 
    gives
	\begin{align*}
		\int_\Omega e^V \frac{\vert u^-\vert^2}{2} ~dx 
		\leq 
		\int_{\Omega} e^V \frac{\vert u_0^-\vert^2}{2} ~dx.
	\end{align*}
	As $u_0$ is assumed to be non-negative this yields the assertion.
\end{proof}

\subsection{Stationary Solutions}
We denote by $\rho_\infty$ the (weak) solution to 
\begin{align}\label{eq:rhostationary}
	 \nabla \cdot \left(- \nabla \rho_\infty + \rho_\infty \nabla V\right) = 0\quad \text{ on } \Omega,
\end{align}
supplemented with boundary conditions \eqref{BoundaryCondition1Ain}--\eqref{BoundaryCondition1Aisolated}. Our first result is the following.
\begin{lemma}
	\label{stationarySolutionBounded}
	Any stationary solution $\rho_\infty \in H^1(\Omega)$ is bounded and strictly positive, i.e. there exists constants $C_0>0$ and $C>0$, such that  $C_0 < \rho_\infty \leq C$ for a.e. $x \in \Omega$.
\end{lemma}
\begin{proof}
	As $V$ is bounded, $\rho_\infty$ is bounded if and only if the stationary Slotboom-variable $u_\infty$, which satisfies  
	\begin{align}
		\label{1AweakformulationSlotboomStationary}
		\int_\Omega e^V\nabla u_\infty \cdot \nabla \phi ~ dx 
		+ \int_{\Gamma_{\text{out}}} \beta u_\infty e^V \phi ~ d\sigma
		= \int_{\Gamma_{\text{in}}} \alpha \phi ~ d\sigma
	\end{align}
	for every $\phi \in H^1(\Omega)$, is bounded.
	First we choose $\phi \equiv 1$ in \eqref{1AweakformulationSlotboomStationary} obtaining
	\begin{align}
		\label{TestFunctionTrivial}
		\int_{\Gamma_{\text{out}}} \beta u_\infty e^V ~ d\sigma
		= \int_{\Gamma_{\text{in}}} \alpha ~ d\sigma.
	\end{align}
	Next we choose $\phi = ( u_\infty -C)_+$ in \eqref{1AweakformulationSlotboomStationary} for some arbitrary constant $C>0$ and obtain after adding an appropriate trivial term
	\begin{gather*}
		\int_\Omega e^V\vert \nabla (u_\infty-C)_+ \vert ^2 ~ dx 
		+ \int_{\Gamma_{\text{out}}} \beta (u_\infty -C)_+^2 e^V ~ d\sigma
		+ C \int_{\Gamma_{\text{out}}} \beta e^V ( u_\infty -C)_+ ~d\sigma \\
		= \int_{\Gamma_{\text{in}}} \alpha (u_\infty -C)_+ ~ d\sigma.
	\end{gather*}
	Applying Friedrich's inequality from theorem \ref{FriedrichsInequality} then results in
	\begin{align}
		\label{Inequality1}
		\frac{K}{C_F} \int_\Omega e^V (u_\infty -C)^2_+ ~ dx 
		+ C \int_{\Gamma_{\text{out}}} \beta e^V ( u_\infty -C)_+ ~d\sigma 
		\leq \int_{\Gamma_{\text{in}}} \alpha (u_\infty -C)_+ ~ d\sigma,
	\end{align}
	with $K = \min \lbrace 1, \beta \rbrace$.
	As $\alpha$ is assumed to be bounded, we can reformulate the right hand side of this inequality by using (weighted) Cauchy's inequality, see \cite[p. 622]{EvansPartialDifferentialEquations}, with constant $1/2\gamma$ and using the trace inequality for $H^1$-functions, see \cite[5.5 Thm.1]{EvansPartialDifferentialEquations}, to estimate
	\begin{gather}
		\int_{\Gamma_{\text{in}}} \alpha ( u_\infty -C) ~ d\sigma
		\leq \int_{\Gamma_{\text{in}}} \alpha_{\text{max}} ( u_\infty -C) ~ d\sigma
		\leq \frac{\alpha}{2\gamma} \int_{\Gamma_{\text{in}}} \alpha_{\text{max}} ~ d\sigma
		+ \frac{\gamma}{2} \int_{\Gamma_{\text{in}}} (u_\infty -C)_+^2 ~ d \sigma 
		\notag
		\\
		\leq \frac{\alpha_{\text{max}}}{2\gamma} \int_{\Gamma_{\text{in}}} \alpha_{\text{max}} ~ d\sigma
		+ \frac{C_{\text{trace}}\gamma}{2} \int_\Omega (u-C)_+^2 ~ dx.
		\label{Inequality2}
	\end{gather}
	Combining \eqref{Inequality1} and \eqref{Inequality2} and using \eqref{TestFunctionTrivial} we gain
	\begin{align*}
		\underbrace{ \Big( \frac{K}{C_F} - \frac{C_{\text{trace}}\gamma}{2} \max \lbrace e^{-V} \rbrace \Big) }_{A}  \int_\Omega e^V ( u_\infty - C)^2_+ ~ dx
		\\
		+ \underbrace{ C \int_{\Gamma_{\text{out}}} \beta e^V (u_\infty -C)_+ ~ d\sigma
		- \frac{\alpha}{2\gamma} \int_{\Gamma_{\text{out}}} \beta u_\infty e^V ~ d\sigma}_{B}
		\leq 0.
	\end{align*}
	If we choose $\gamma$ large enough so that $A$ is positive and $C \geq \frac{\alpha}{2 \gamma}$ such that $B$ is positive, we can omit $B$ obtain that a.e. $u_\infty \leq C$. 

To show that the stationary solution is strictly positive, we first note that the argument of Lemma \ref{1AWeakSolNonNeg} also holds for the stationary solution and thus $u_\infty \in L^\infty(\Omega)$. Together with $u \in H^{1}(\Omega)$, this allows us to apply \cite[Theorem 4]{LeHalStrongPositivity} to \eqref{1AweakformulationSlotboomStationary} to conclude strict positivity.
% 	in strong form and note that $u \in W^{1,2}(\Omega) \cap L^{\infty}(\Omega)$ again transform to the Slotboom-variable and achieve
% 	\begin{alignat*}{2}
% 		- \nabla \cdot (e^V \nabla u) &= 0 \ \ &&\text{ on } \Omega, \\
% 		e^V \nabla u \cdot n &= \alpha  \ \ &&\text{ on } \Gamma_{\text{in}}, \\
% 		e^V \nabla u \cdot n + \beta ue^V &=0  \ \ &&\text{ on } \Gamma_{\text{out}}. 
% 	\end{alignat*}
% 	As $\alpha$ and $\beta$ are non-negative functions, using a similar argument as in lemma \ref{1AWeakSolNonNeg} $u$ is also non-negative. Then, the boundedness of $u$ shown above implies its strict positivity, see \cite[Theorem 4]{LeHalStrongPositivity}, as .
\end{proof}

Next we show that such a stationary solution actually exists and that it is unique, closely following the proof of Proposition 4.1 in \cite{BurgerPietschmannFlowCharacteristics}.

\begin{lemma}[Existence of the stationary solution]
	\label{1AStationarySolutionStrictlyPositive}
	Let (A1)-(A5) hold, then there exists exactly one stationary solution $\rho_\infty \in H^1(\Omega)$ to \eqref{eq:rhostationary} with boundary conditions \eqref{BoundaryCondition1Ain}--\eqref{BoundaryCondition1Aisolated}.
\end{lemma}
\begin{proof}
	Using the G\aa rding inequality \eqref{eq:garding}, existence follows from the standard theory for elliptic equations, see \cite[Section 6.2.]{EvansPartialDifferentialEquations}. Now let $\rho_1$ and $\rho_2$ be two stationary solutions, then 
	$\omega = \rho_1  - \rho_2$ satisfies
	\begin{align*}
		0 &= \nabla \cdot (- \nabla \omega  + \omega \nabla V) \text{ on } \Omega
	\end{align*}
	with boundary conditions
	\begin{alignat*}{3}
		-\nabla \omega  + \omega \nabla V &= 0 &\text{ in } &\Gamma_{\text{in}} \text{ and} \\
		-\nabla \omega  + \omega \nabla V &= \beta \omega &\text{ in } &\Gamma_{\text{out}}.		
	\end{alignat*}
	Now let $\nu = \omega e^{-V}$, then $\nu$ is the weak solution of $\nabla \cdot ( e^V \nabla \nu) = 0$ in $\Omega$ with boundary conditions 
	\begin{alignat*}{3}
		e^V \nabla \nu &= 0 &\text{ in } &\Gamma_{\text{in}} \text{ and} \\
		e^V \nabla \nu &= - \beta \nu &\text{ in } &\Gamma_{\text{out}}.
	\end{alignat*}
	Using the weak formulation of this boundary value problem with test function $\nu$ implies
	\begin{align*}
		- \int_\Omega e^V \vert \nabla \nu \vert^2 ~ dx 
		- \int_{\Gamma_{\text{out}}} \beta e^V \nu^2 ~dx = 0,
	\end{align*}
	which yields $\nu = 0 = \omega$ (using Friedrichs' inequality with boundary values) and thus uniqueness of the solution.
\end{proof}

\subsection{Entropy Dissipation with Logarithmic Entropy}
The standard approach to prove exponential convergence of $\rho(x,t)$ as $t\to \infty$ would be choosing a logarithmic entropy functional and using a logarithmic-Sobolev inequality to bound the dissipation by the relative entropy. But this approach fails in our case, due to the lack of mass conservation. More precisely, a scaling argument shows that the desired logarithmic-Sobolev inequality fails to hold. Indeed, consider
% 
% 
% Entropy methods provide various applications in the understanding of the qualitative behavior of partial differential equations.
% One of them is the large-time behavior of solutions or more precise convergence of the partial differential equation to some equilibrium state.
% To analyze the long time behavior of equation \eqref{Fokker-Planck1A} we at first choose the logarithmic relative entropy 
\begin{align}\label{eq:relativelogentropy}
	E(\rho\vert \rho_\infty) = \int_\Omega \rho  \log\Big(\frac{\rho}{\rho_\infty}\Big) - ( \rho - \rho_\infty) ~ dx.
\end{align}
Calculating the dissipation of this functional yields
\begin{align*}
	D(\rho \vert \rho_\infty) 
	&\coloneqq - \frac{d}{dt} E(\rho\vert \rho_\infty)
	= - \int_\Omega \partial_t \rho \cdot \log\Big(\frac{\rho}{\rho_\infty}\Big) ~ dx 
	= \int_\Omega \nabla \cdot J \log\Big(\frac{\rho}{\rho_\infty}\Big) ~ dx \\
	&= 
	\underbrace{-\int_\Omega J \nabla \log\Big(\frac{\rho}{\rho_\infty}\Big) ~ dx}_{\substack{A}}
	- \int_{\Gamma_{\text{in}}} \alpha \log\Big(\frac{\rho}{\rho_\infty}\Big) ~ d\sigma + \int_{\Gamma_{\text{out}}} \beta\rho e^V \log\Big(\frac{\rho}{\rho_\infty}\Big) ~ d\sigma. 
	\\
	\intertext{We now use $J=\rho \nabla(-\log(\rho)+V)$, to write $A$ as}
	A
	&= \int_\Omega \rho \nabla\Big(\log\big(\frac{\rho}{\rho_\infty}\Big)+\log(\rho_\infty)-V\Big)\nabla \log\Big(\frac{\rho}{\rho_\infty}\Big) ~ dx \\
	%&= \int_\Omega \rho \big\vert\nabla\log\big(\frac{\rho}{\rho_\infty}\big)\big\vert^2 ~ dx
	%+ \int_\Omega \rho \nabla \big((\log(\rho_\infty)-V\big)\nabla \log\Big( \frac{\rho}{\rho_\infty} \Big) ~ dx 
	%\\
	&= \int_\Omega \rho \big\vert \nabla\log\big(\frac{\rho}{\rho_\infty}\big)\big\vert^2 ~ dx
	- \int_\Omega \rho \frac{J_\infty}{\rho_\infty}\nabla \log\Big( \frac{\rho}{\rho_\infty} \Big) ~ dx \\
	&= \int_\Omega \rho \big\vert \nabla\log\big(\frac{\rho}{\rho_\infty}\big)\big\vert^2 ~ dx
	- \int_\Omega J_\infty \nabla \frac{\rho}{\rho_ \infty} ~ dx \\
	&= \int_\Omega \rho \big\vert \nabla\log\big(\frac{\rho}{\rho_\infty}\big)\big\vert^2 ~ dx
	%+ \int_\Omega \underbrace{\nabla J_\infty\cdot n}_{\substack{=0}} \frac{\rho}{\rho_ \infty} ~ dx 
	%\\&\qquad 
	+  \int_{\Gamma_{\text{in}}} \alpha \frac{\rho}{\rho_\infty} ~ d\sigma 
	- \int_{\Gamma_{\text{out}}} \beta\rho_\infty e^V \frac{\rho}{\rho_\infty} ~ d\sigma, 
\end{align*}
where we used $u\nabla \log(u) = \nabla u$ in the penultimate line and Gauss' theorem in the last equation.
To show the entropy entropy dissipation inequality we need the following lemma: 
\begin{lemma}
	\label{Hilfslemmaintegralistpositiv1A}
	The following sum of integrals is positive:
	\begin{align*}
		B=- \int_{\Gamma_{\text{in}}} \alpha \log\Big(\frac{\rho}{\rho_\infty}\Big) ~ d\sigma + \int_{\Gamma_{\text{out}}} \beta\rho e^V \log\Big(\frac{\rho}{\rho_\infty}\Big) ~ d\sigma \\ 
		+  \int_{\Gamma_{\text{in}}} \alpha \frac{\rho}{\rho_\infty} ~ d\sigma - \int_{\Gamma_{\text{out}}} \beta\rho_\infty e^V \frac{\rho}{\rho_\infty} ~ d\sigma .
	\end{align*}
\end{lemma}
\begin{proof}
	The function $F(t)= t \log (t)$ is convex for $t>0$, so $F(t)-F(s)-F'(s)(t-s) \geq 0$.
	Thus we have $t \log \frac{t}{s} - t +s \geq 0$, which we denote by ($*$) for the future.
	We want to add the following sum onto $B$ which is zero because of Gauss's theorem:
	\begin{align*}
		 -\int_{\Gamma_{\text{in}}} \alpha ~ d\sigma + \int_{\Gamma_{\text{out}}} \beta e^V \rho_\infty ~ d\sigma	
		= - \int_{\partial \Omega} J^{\infty} \cdot n ~ dx 
		= - \int_{\Omega} \nabla J^{\infty} ~ dx  = 0.
	\end{align*}
	So we gain by addition of zero, ($*$) and $\log(x)-x+1 \leq 0$ for $x\geq0$ the following:
	\begin{align*}
		B 
		&= - \int_{\Gamma_{\text{in}}} \alpha \Big( \log \frac{\rho}{\rho_\infty} -\frac{\rho}{\rho_\infty} \Big) ~ d\sigma
		- \int_{\Gamma_{\text{out}}} \beta e^V \Big(\rho -\rho \log \frac{\rho}{\rho_\infty}\Big) ~ d\sigma	\\
		&\qquad - \int_{\Gamma_{\text{in}}} \alpha dx + \int_{\Gamma_{\text{out}}} \beta e^V \rho_\infty ~ d\sigma \\
		&= - \int_{\Gamma_{\text{in}}} \alpha \Big( \log \frac{\rho}{\rho_\infty} -\frac{\rho}{\rho_\infty} +1\Big) ~ d\sigma
		+ \int_{\Gamma_{\text{out}}} \beta e^V \Big(\rho \log \frac{\rho}{\rho_\infty} -\rho +\rho_\infty\Big) ~ d\sigma \geq 0.		
	\end{align*}
\end{proof}

Applying Lemma \ref{Hilfslemmaintegralistpositiv1A} gives us
\begin{align*}
	D(\rho \vert \rho_\infty) 
	&\geq \int_\Omega \rho \big\vert\nabla\log\big(\frac{\rho}{\rho_\infty}\big)\big\vert^2 dx 
	= \int_\Omega \rho ~ \frac{\rho_\infty}{\rho} ~ \Big\vert \frac{\nabla \frac{\rho}{\rho_\infty}}{ \sqrt{\frac{\rho}{\rho_\infty}}} \Big\vert^2 dx
	= 4 \int_\Omega \rho_\infty \vert \nabla \big(\sqrt{ \frac{\rho}{\rho_\infty}}\big)\vert^2 dx.
\end{align*}
If we define $\phi = \sqrt{\frac{\rho}{\rho_\infty}}$, we would need the following inequality, weighted by the strictly positive function $\rho_\infty$,
\begin{align*}
	\int_\Omega \vert \nabla \phi\vert^2dx 
	\leq C \Big(\int_\Omega \phi^2 \log(\phi^2)- \phi^2 + 1 dx - \int_{\Gamma_{\text{out}}} \phi^2 \log(\phi^2) - \phi^2 +1 ~ d\sigma\Big)
\end{align*}
to gain the desired result.
Surprisingly this inequality is badly scaled.
To see this, assume that there exists a function $\phi$ having trace zero on $\Gamma_{\text{out}}$ that satisfies the inequality. Replacing $\phi$ by $K\phi$ with constant $K \ge 0$
% Thus $\psi \coloneqq K\phi$ with $K \in \R$ should also fulfill the inequality. But inserting $\psi$ in the inequality 
gives
\begin{align*}
	\int_\Omega \vert \nabla \phi\vert^2 ~dx &\leq C 
	\Big( \int_\Omega \phi \log(K\phi)^2 - \phi^2 + \frac{1}{K^2} ~dx
	- \int_{\Gamma_{\text{out}}} \phi^2 \log(K\phi)^2 - \phi^2 + \frac{1}{K^2} ~d\sigma \Big).
\intertext{As $\phi$ has zero boundary values at $\Gamma_{\text{out}}$, we obtain}
	\int_\Omega \vert \nabla \phi\vert^2 ~dx &\leq C 
	\Big( \int_\Omega \phi \log(K\phi)^2 - \phi^2 + \frac{1}{K^2} ~dx
	- \int_{\Gamma_{\text{out}}} \frac{1}{K^2} ~d\sigma \Big)
\end{align*}
and if we chose $K$ small enough, the inequality becomes wrong.

\subsection{Entropy dissipation with Quadratic Entropy}

\label{SubsectionQuadraticEntropy}
Although the logarithmic entropy is physically more natural, the preceding discussion showed that it is not suitable in our setting. However, as the problem is linear, the quadratic relative entropy defined as 
\begin{align}
	\label{1ARelativeEntropy}
	E(\rho\vert \rho_\infty) 
	= \frac{1}{2} \int_\Omega \frac{(\rho-\rho_\infty)^2}{\rho_\infty} ~ dx
\end{align}
yields the desired exponential decay towards equilibrium. Its dissipation is given as
\begin{align} \label{1Adissipation}
	\begin{split}
	D(\rho\vert \rho_\infty) 
	&= - \int_\Omega \partial_t \rho \, \frac{\rho-\rho_\infty}{\rho_\infty} ~ dx 	
	= - \int_\Omega (- \nabla \rho + \rho \nabla V)\cdot \nabla \Big(\frac{\rho-\rho_\infty}{\rho_\infty}\Big) ~ dx 
	\\
	&\qquad -  \int_{\Gamma_{\text{in}}} \alpha ~ \frac{\rho-\rho_\infty}{\rho_\infty} ~ d\sigma
	+ \int_{\Gamma_{\text{out}}} \beta\rho \, \frac{\rho-\rho_\infty}{\rho_\infty} ~ d\sigma. 
	\end{split}
\end{align}
To proceed we test the equation $-\nabla \cdot J_\infty =0$ with $ \phi = \frac{\rho-\rho_\infty}{\rho_\infty}$ which yields
\begin{align*}
	0
	= \int_\Omega (- \nabla \rho_\infty + \rho_\infty \nabla V)\cdot \nabla \Big(\frac{\rho-\rho_\infty}{\rho_\infty}\Big) ~ dx 
	+ \int_{\Gamma_{\text{in}}} \alpha ~ \frac{\rho-\rho_\infty}{\rho_\infty} ~ d\sigma
	- \int_{\Gamma_{\text{out}}} \beta \rho_\infty  \frac{\rho-\rho_\infty}{\rho_\infty} ~ d\sigma.
\end{align*}
Adding the last two equations then results in
\begin{align}
	D(\rho\vert \rho_\infty) 
	&= \int_\Omega \Big( \nabla (\rho-\rho_\infty) 
	+ (\rho_\infty - \rho ) \nabla V \Big) \cdot \nabla \Big(\frac{\rho-\rho_\infty}{\rho_\infty}\Big) ~ dx
	+ \int_{\Gamma_{\text{out}}} \beta \frac{(\rho-\rho_\infty)^2}{\rho_\infty} ~ d\sigma.
	\notag
	\intertext{Using the definition of $\phi$ then yields}	
	&= \int_\Omega \rho_\infty \vert \nabla \phi \vert^2 
	- ( - \nabla \rho_\infty 
	+ \rho_\infty \nabla V) \cdot \phi\nabla \phi ~ dx
	+  \int_{\Gamma_{\text{out}}} \beta \rho_\infty \phi^2 ~ d\sigma 
	\notag
	\\
	%&= \int_\Omega \rho_\infty \vert \nabla \phi \vert^2 
	%- \frac{1}{2} \nabla \phi^2 \cdot(-\nabla \rho_\infty 
	%+ \rho_\infty \nabla V) ~ dx
	%+ \int_{\Gamma_{\text{out}}} \beta\rho_\infty  \phi^2 ~ d\sigma 
	%\notag
	%\\
	&= \int_\Omega \rho_\infty \vert \nabla \phi \vert^2 ~ dx
	+ \frac{1}{2} \int_\Omega \phi^2 ~ \nabla \cdot J_\infty ~ dx 
	+ \int_{\Gamma_{\text{in}}}  \frac{\alpha}{2} \phi^2 ~ d \sigma
	+ \int_{\Gamma_{\text{out}}} \frac{\beta}{2}  \rho_\infty \phi^2 ~ d \sigma	
	\label{1Adissipationcomparision}
	\\
	&\geq \int_\Omega \rho_\infty \vert \nabla \phi \vert^2 ~ dx
	+ \frac{\beta_0}{2} \int_{\Gamma_{\text{out}}} \rho_\infty \phi^2 ~ d \sigma,
	\notag
\end{align}
where the last inequality holds as $\nabla \cdot J_\infty=0, ~ \beta \geq \beta_0$ and the influx term is positive. With $K_1=\min\lbrace \frac{\beta_0}{2}, 1\rbrace$ and $C_F$ being the Friedrich's constant, we gain
\begin{align}
	D(\rho\vert \rho_\infty)\geq \frac{1}{2} C_F K_1 \int_\Omega \rho_\infty \vert \phi \vert^2 ~ dx
	= C_F K_1 E(\rho \vert \rho_\infty),
	\label{1Asigndissipation}
\end{align}
as $\rho_\infty$ is strictly positive as proven in lemma \ref{1AStationarySolutionStrictlyPositive}.
Combining this inequality with Gronwall's lemma, we obtain 
\begin{align*}
	E(\rho \vert \rho_\infty) \leq E(\rho_0 \vert \rho_\infty)e^{- C_F K_1 t}
\end{align*}
and due to the Cauchy-Schwarz inequality the desired exponential decay in $L^1(\Omega)$ via
\begin{align*}
	\Vert \rho - \rho_\infty \Vert^2_{L^1(\Omega)} 
	= \Big( \int_\Omega \vert \frac{\rho} {\rho_\infty} - 1 \vert ~ \rho_\infty ~ dx \Big) \leq 2 \max \lbrace \rho_\infty \rbrace E(\rho\vert \rho_\infty).
\end{align*}

\subsection{Numerical Solutions}\label{sec:Numerik1A}
Finally we want to solve the one-dimensional problem numerically.
After an appropriate scaling we assume $\Omega = [0,1]$ with $\partial\Omega=\lbrace 0,1 \rbrace$,
so the one-dimensional version of \eqref{Fokker-Planck1A} becomes
\begin{align}
	\label{1A1Dversion}
	0 = \partial_t \rho 
	+ \partial_x \big( - \partial_x \rho + \rho \partial_x V \big) \text{ on } (0,1) \times (0,T),
\end{align}
with boundary condition 
\begin{alignat*}{3}
	J &= \alpha			&\text{ at } &x = 0 \text{ and} \\
	J &= \beta \rho(1)	&\text{ at } &x = 1.
\end{alignat*}
Note that in this setting we can give an explicit characterization of the stationary solution.
As the flux rates $\alpha$ and $\beta$ are constants we have $J = -\partial_x \rho_{\infty} + \rho_{\infty} \partial_x V = \alpha$ because of the boundary condition at $x = 0$.
Solving this ordinary differential equation, we obtain the strictly positive stationary solution
\begin{align}
	\rho_{\infty}(x) 
	&= \Big( -\int_0^x \alpha e^{-V(x)} dx + C \Big) e^{V(x)} \geq \frac{\alpha}{\beta}e^{V(1)} > 0,
		\label{1Astationarysolution1D}
\intertext{with the constant}
	C 
	&= \alpha \Big( \frac{e^{-V(1)}}{\beta} + \int_0^1 e^{-V(x)}dx \Big).
	\notag
\end{align}
For simplicity, we chose $V(x)=x$ in the following, i.e. the simplest case in which mass is transported from the left entrance to the right exit. We introduce an uniformly spaced grid with $n$ grid points $x_0,..,x_{n-1}$, with $x_0=0$ and $x_{n-1}=1$ and discretize \eqref{1A1Dversion} using a fully explicit finite difference scheme. The boundary conditions are implemented by introducing two fictitious nodes $x_{-1}$ and $x_{n}$ outside of the physical domain.
The algorithm is implemented in MATLAB, choosing the number of grid points $n = 200$ and $dt = 5 \cdot 10^{-6}$.
% 
% \begin{remark}
% 	Having the geometry of the one-dimensional version of \eqref{Fokker-Planck1A} in mind there is influx at $x=0$ and outflux at $x=1$.
% 	Therefore the in- and outflux terms transport particles from 0 to 1.
% 	
% 	If we choose $V(x)=x$ as the potential, the part of the Fokker-Planck equation which is responsible for the drift evolves to $\partial_t \rho + \partial_x \rho = 0$,
% 	which is exactly the linear transport equation.
% 	Its solution is $u(x,t)=u_0(x-t)$ and therefore it transports the particles in the same direction as the in- and outflux terms, i.e. from 0 to 1 or less formally to the `right' of the domain.	
% 	Whereas if we choose $V(x)=-x$ the drift term works against the influx and the outflux leading to a smaller convergence speed.
% \end{remark}

\subsubsection{Time Evolution of the Density}
In Figure \ref{1AResultNumericalSolution} the time evolution of the particle concentration $\rho$ solving \eqref{Fokker-Planck1A} is compared to the stationary solution $\rho_\infty$ for $\alpha =1, \beta=0.9$ and initial concentration $\rho_0(x) = -0.1x +1.2$. 
For our choice $V(x)=x$ the stationary solution is of the form
\begin{align}
	\rho_{\infty}(x) = \alpha  + \Big( \frac{1}{\beta e}- \frac{1}{e}\Big)\alpha e^{x}.
	\label{1AstationarysolutionVx}
\end{align}
In Figure \ref{1AResultNumericalSolution} (a) the direct comparison between the initial function $\rho_0(x) = -0.1x +1.2$ and the stationary solution is shown.
We chose this initial concentration as its shape is very different to the stationary solution, i.e. the gradient has the opposite sign. 
In Figure \ref{1AResultNumericalSolution} (b) the influence of the boundary conditions is strongly visible. 
At $x=0$ particles have entered the domain and at $x=1$ particles have left the domain.
This effect is combined with the drift which leads to a maximum in the left part of the domain and to a minimum at the right.
In Figure \ref{1AResultNumericalSolution} (c) one sees that after some time the shape of the solution is similar to the one of the stationary solution and only the low frequency parts need longer to converge.
Lastly Figure \ref{1AResultNumericalSolution} (d) shows this long time behavior and not surprisingly there is no difference visible between $\rho$ and $\rho_\infty$ anymore.

\begin{figure}[h]
	\begin{center}
		\begin{tabular}{cc}
\includegraphics[width=0.40\textwidth]{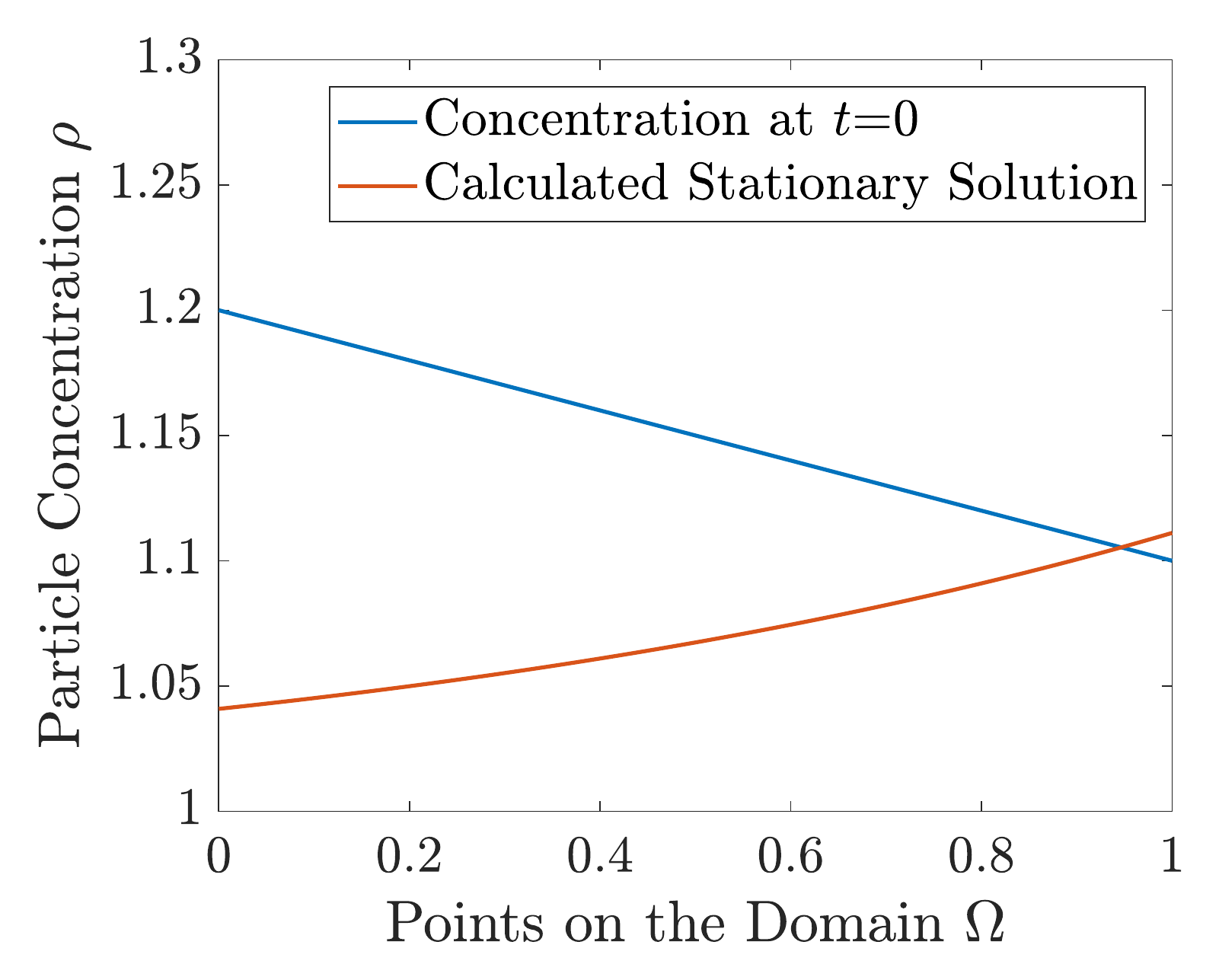} &
\includegraphics[width=0.40\textwidth]{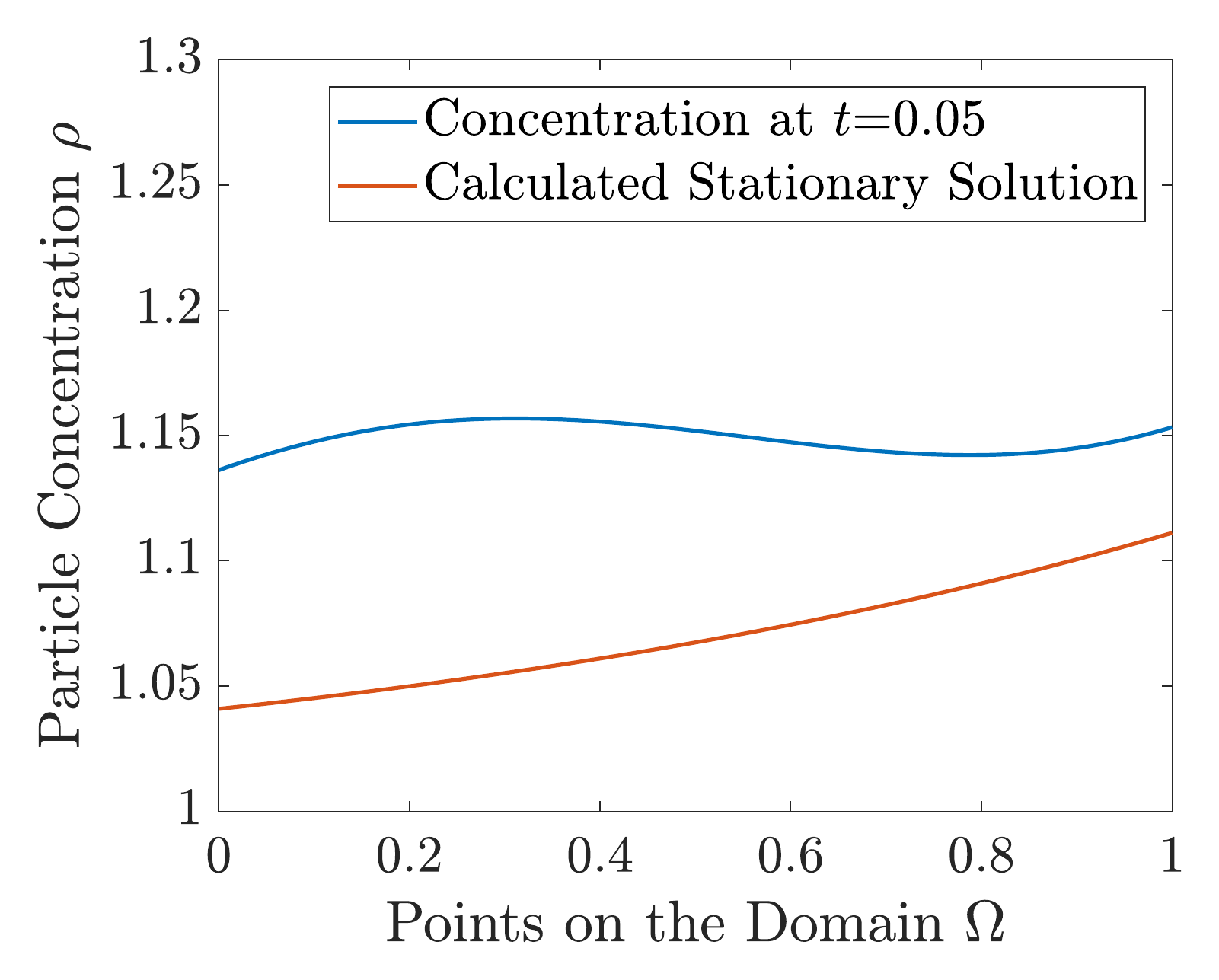} \\
		(a) & (b) \\
\includegraphics[width=0.40\textwidth]{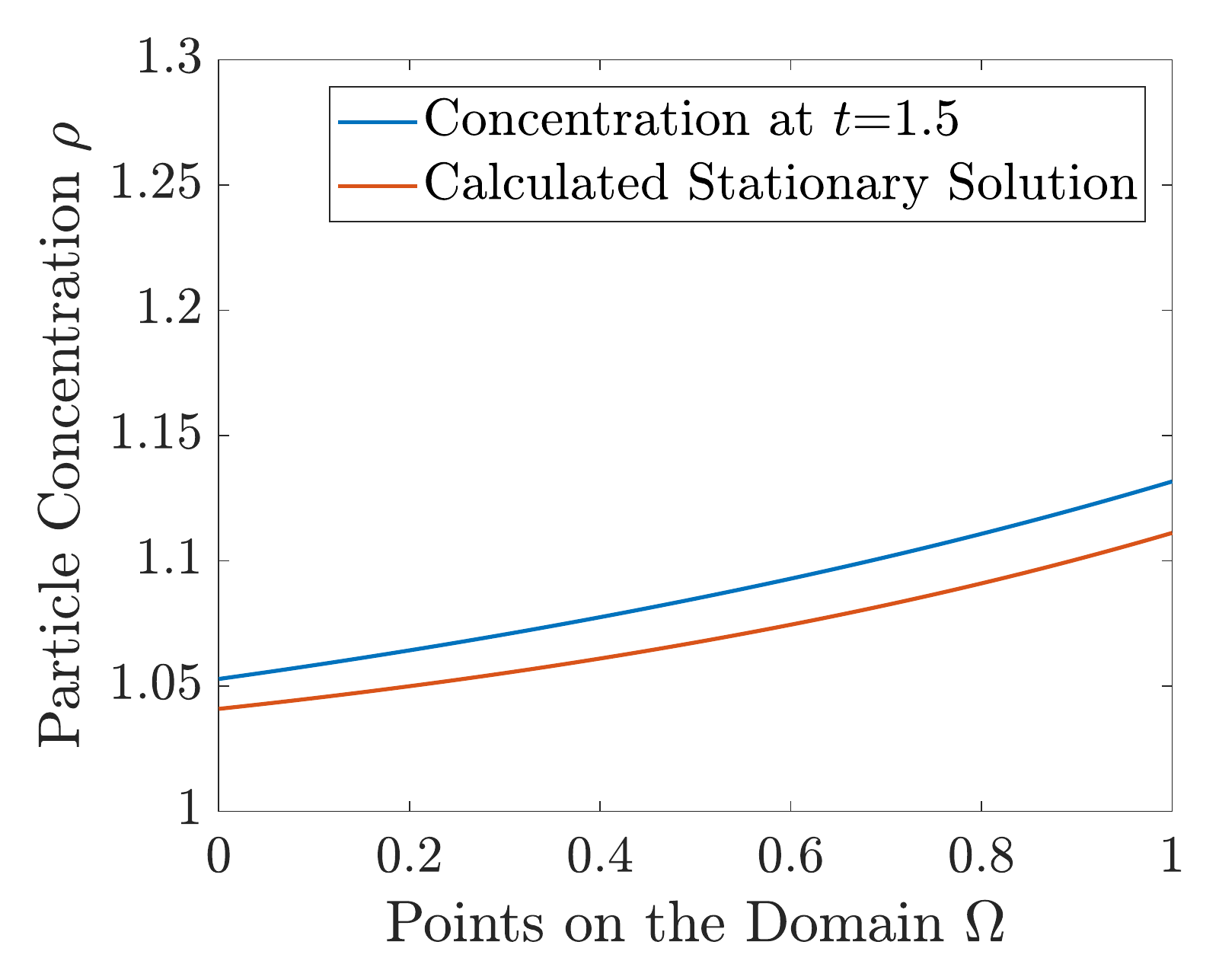} &
\includegraphics[width=0.40\textwidth]{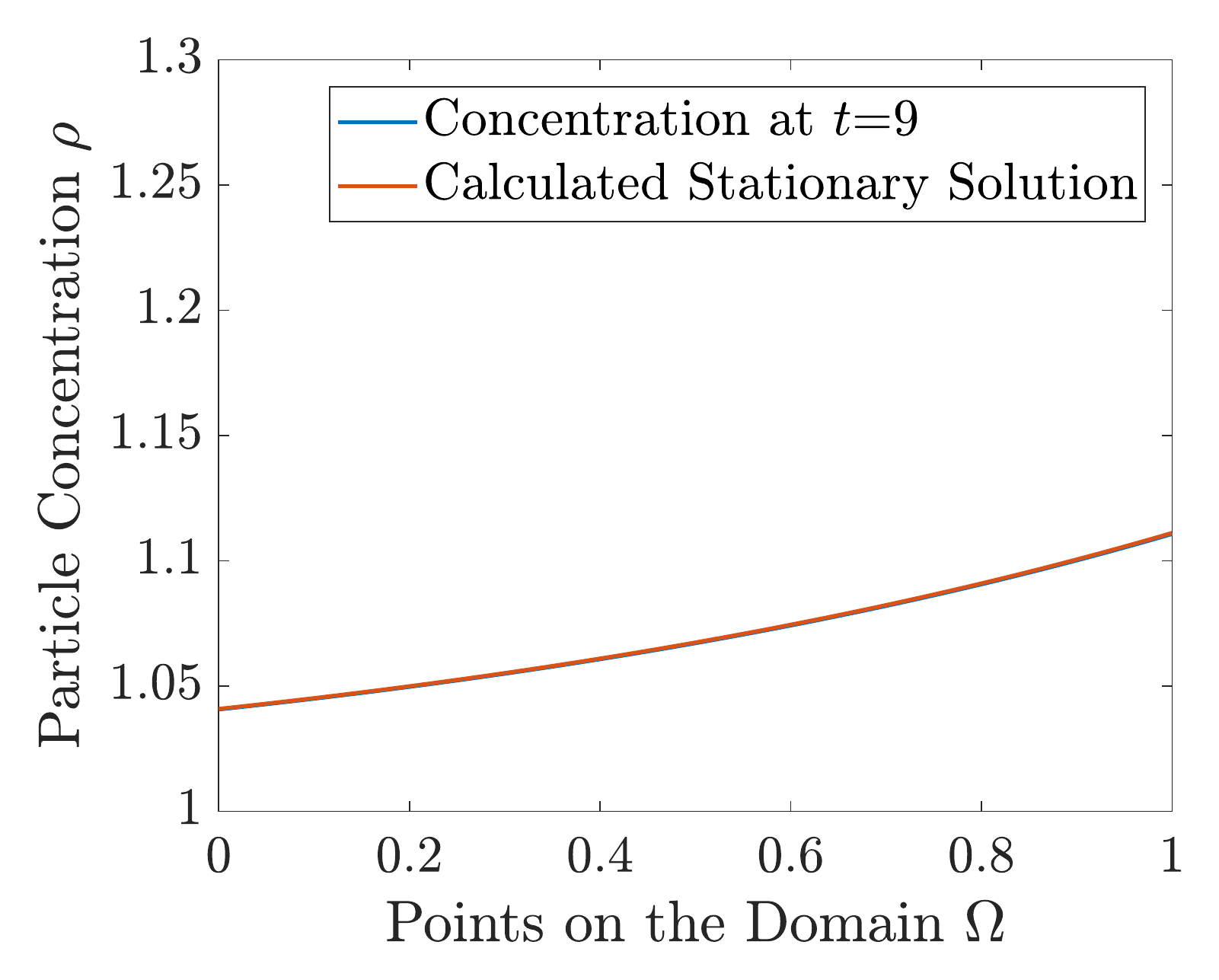} \\
		(c) & (d)
		\end{tabular}
	\end{center}
	\caption{\textbf{Evolution over Time of the Particle Concentration:} The time evolution of $\rho$ solving \eqref{Fokker-Planck1A} in comparison to the calculated stationary solution \eqref{1AstationarysolutionVx} for $\alpha =1, \beta=0.9$ and initial particle concentration $\rho(x) = -0.1x +1.2$.
	(a) The initial concentration at $t=0$, (b) strong influence of the boundary terms at $t=0.05$, (c) strong influence of the drift term and the diffusion at $t=1.5$, (d) equilibrium state at $t=9$.}
	\label{1AResultNumericalSolution}
\end{figure}

\subsubsection{Convergence Rates for the Relative Entropy}
Next we compare the numerical rate of convergence to the analytical results of section \ref{SubsectionQuadraticEntropy}. To this end we chose $\alpha=\beta=1$ which yields $\rho_\infty = 1$. Starting again with initial concentration $\rho_0(x) = -0.1x +1.2$, we observe exponential convergence with rate $m_n=-2.33$ (corresponding to the case $\gamma = 1$ in Figure \ref{1AEntropyPlot}).\\
Now we examine the calculations of section \ref{SubsectionQuadraticEntropy}.
Except for the application of Friedrich's inequality, all other manipulations are equalities. Thus, comparing \eqref{1Adissipationcomparision} and \eqref{1Asigndissipation}, we see that the analytic rate of convergence $m_a$ is determined by the constant in Friedrich's inequality, i.e.
\begin{align}
	\frac{1}{2} m_a =  \inf_{\phi} \frac{\int_0^1 (\phi')^2 ~dx + \frac{1}{2}\phi(0)^2 + \frac{1}{2}\phi(1)^2}{\int_0^1 \phi^2 ~ dx}
	=: \lambda.
	\label{1AQuotientToCalculateF-Constant}
\end{align}
To solve this minimization problem we define the functional
\begin{gather*}
	V(\phi) 
	= 
	\lambda \int_0^1 \phi^2 ~ dx  
	- \int_0^1 (\phi')^2 ~ dx 
	- \frac{1}{2}\phi(0)^2 
	- \frac{1}{2} \phi(1)^2, \ \ \text{ where}
	\\
	D_\psi V(\phi) =  2 \int_0^1 \Big( \phi \lambda + \phi'' \Big) \psi~ dx
	- 2\psi(1) \Big( \frac{1}{2} \phi(1) + \phi'(1) \Big) 
	- 2\psi(0) \Big(\frac{1}{2} \phi(0) - \phi'(0)  \Big) 
\end{gather*}
is the functional G\^{a}teaux-derivative in direction $\psi$ after integration by parts.
It is zero if the three conditions 
$$\phi'' = - \lambda \phi, ~ \frac{1}{2}\phi(1) + \phi'(1) = 0\text{ and }\frac{1}{2}\phi(0) - \phi'(0) = 0$$
hold. The function $\phi(x) = a \sin(kx) + b \cos(kx)$ satisfies the first condition if $\lambda = k^2$.
The third condition gives $b = 2 a k$ and for simplicity we chose $a =1$.
Thus the second condition translates to 
\begin{align*}
	2 k \cos (k) + \big(0.5- 2k^2\big)\sin(k) = 0,
\end{align*}
where a numerical calculation shows that the smallest and positive $k$, which solves this equation is $k \approx 0.9602$.
Inserting this in the first condition gives $m_a = 2\lambda = \inf 2 k^2 \approx 1.8439$.
So analytically the relative entropy can be bounded above by $\tilde{f}(x)=E(\rho_0\vert \rho_\infty) e^{-1.8439t}$.

Surprisingly this value differs significantly from the numerically calculated slope $m_n = -2.33$. This can be explained by the fact that the operator
\begin{align}
	A[\rho, \psi] 
	= \int_0^1 \rho'\psi' + (\rho ~ V')'\psi ~ dx
	+ \beta \rho(1)\psi(1)
	\label{1AOperatorAnalyticalCalculation}
\end{align}
is symmetric except for the drift term which is skew-symmetric. For any symmetric operator $A$ the spectral gap $\lambda$ determines the slowest possible convergence rate as
\begin{align*}
	u'=-Au + f \quad \Leftrightarrow \quad u(t) \leq \Big(u_0 + \int_0^1 f ~ dx \Big) e^{-\lambda t},
\end{align*}
by Gronwall's lemma. The skew-symmetric part can however mix the eigenvalues which can result in a faster rate of convercence. To examine this phenomena in more detail, we consider only the the symmetric part of the operator \eqref{1AOperatorAnalyticalCalculation}. Its eigenvalue is determined by 
\begin{gather*}
	\int_0^1 \rho'\psi' dx + \beta \rho(1)\psi(1) = \lambda (\rho,\psi) \\
	\Rightarrow \qquad -\int_0^1 (\rho'' + \lambda \rho)\psi ~ dx
	+ \psi(1)\Big(\beta\rho(1)+\rho'(1)\Big) - \rho'(0)\psi(0) = 0.
\end{gather*}
Elementary calculations yield that the function $u(x) = \cos(kx)$ is an eigenfunction with smallest eigenvalue $\lambda = k^2= 0.8603^2$ where $k$ is the solution of $-k^2 \cos(kx) + \lambda \cos(kx) =0$. Thus the convergence rate of the symmetric problem is $m_s= 2\lambda= 1.4802$, as calculating the infimum in \eqref{1AQuotientToCalculateF-Constant} is equivalent to calculating the eigenvalue of the corresponding operator. Indeed, this is confirmed by our numerical calculations in the case $V=0$.
% Thus, in the symmetric case the smallest eigenvalue of the operator corresponds to the numerically calculated smallest convergence speed possible, see Figure \ref{1AEntropyPlot} (a). 

The entropy dissipation calculation in section \ref{SubsectionQuadraticEntropy} on the other hand is insensitive to the skew-symmetric part of the operator because we used integration by parts to gain \eqref{1Adissipationcomparision} and to get rid of the potential-term. This explains the deviation between $m_n$ and $m_a$ as soon as $V\neq 0$.
% In contrast to that, the numerical calculation respects both the symmetric and the skew-symmetric part of the operator.
% Therefore leading to two different convergence speeds in the skew-symmetric case. 

Finally, we analyze numerically how the strength of the potential term influences the convergence rate. We consider
\begin{align*}
	0 = \partial_t \rho + \nabla \cdot (- \nabla \rho + \gamma \rho \nabla V)
\end{align*}
for different values of $\gamma \le 0$ and the same boundary conditions as in the initial situation, i.e. \eqref{BoundaryCondition1Ain} - \eqref{BoundaryCondition1Aisolated}.
This modification of the PDE effects the shape of the relative entropy and the (non exponential) relation between the scaling factor and the rate of convergence is shown in Figure \ref{1AEntropyPlot} (c).

\begin{remark}
	Clearly $\alpha$ has no influence on the convergence rate as it is not part of the operator \eqref{1AOperatorAnalyticalCalculation}.
	It just influences the constant $E(\rho_0\vert\rho_\infty)$ of the relative entropy as $\alpha$ influences $\rho_\infty$. The outflux term $\beta$ influences both the convergence rate and the constant.
\end{remark}

\begin{remark}
    Note that depending on the choice of initial datum, the mass evolution due to the boundary conditions may dominate over the exponential convergence for short times. Clearly, asymptotically, one observes the exponential rate shown in theorem \ref{1AEexpDecay}. This also motivates the proceeding subsection in which we study the mass evolution in more detail.
\end{remark}

\begin{figure}[t]
	\begin{center}
		\begin{tabular}{cc}
		\includegraphics[width=0.47\textwidth]{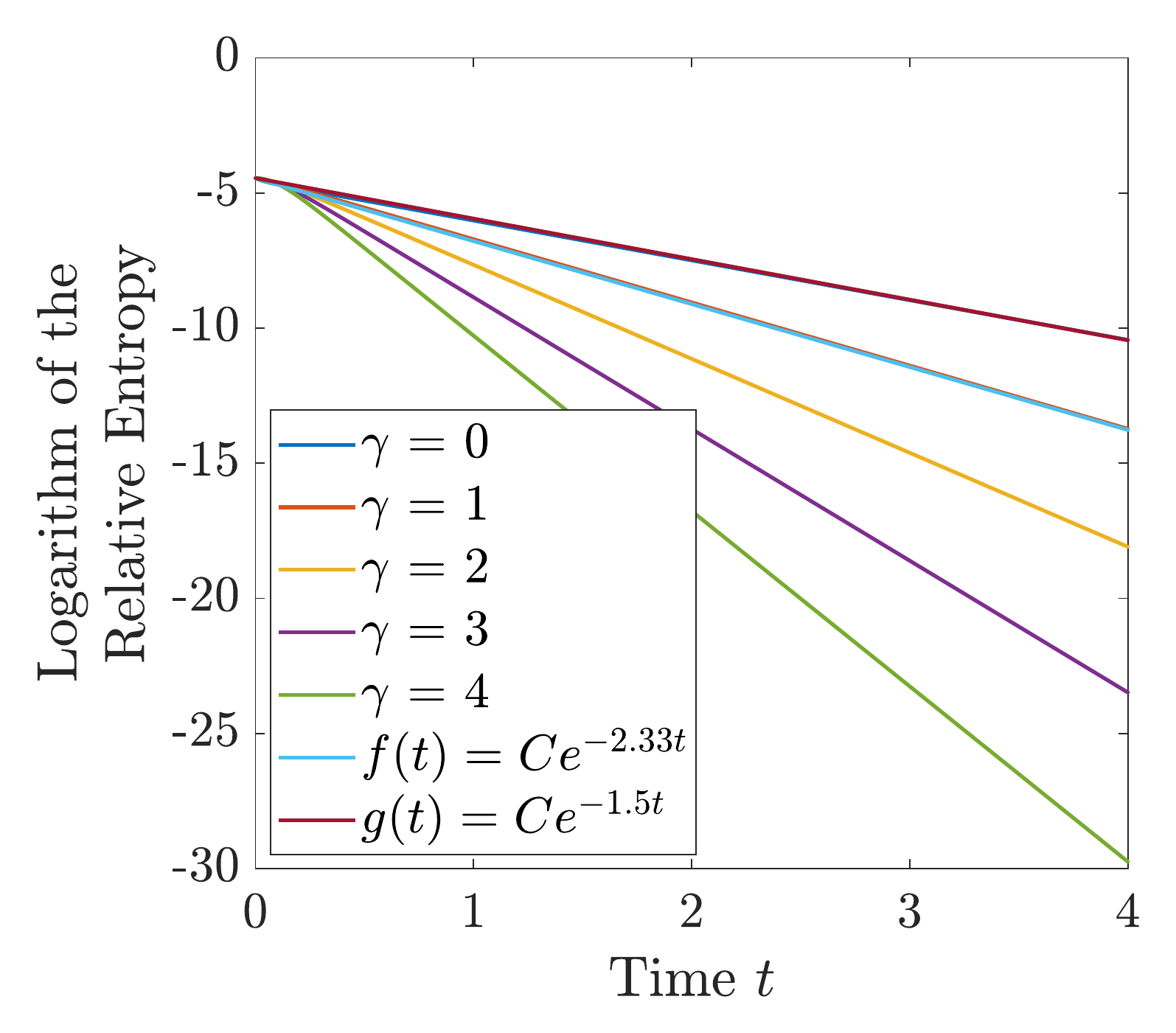} &
		\includegraphics[width=0.47\textwidth]{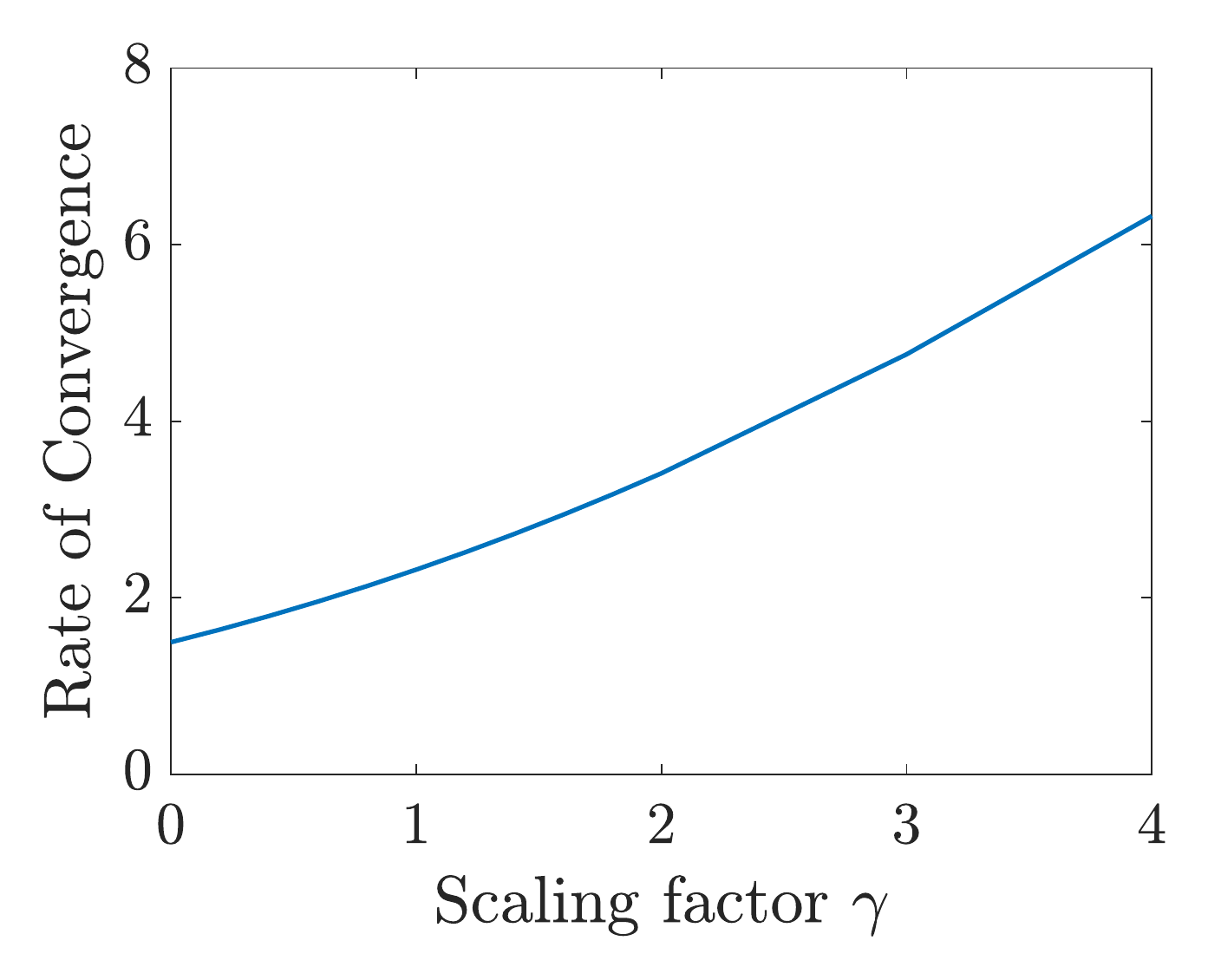}
		\\
		(a) & (b)
		\end{tabular}
	\end{center}
	\caption{\textbf{Relative Entropy:} The relative entropy \eqref{1ARelativeEntropy} for the one dimensional version of \eqref{Fokker-Planck1A} for the initial particle concentration $\rho(x) = -0.1x +1.2$. (a) Natural logarithm of the relative entropy for $\alpha =1, ~ \beta=1$ and variable scaling factor $\gamma$ of the potential term, (b) relation between the scaling factor $\gamma$ and the corresponding slope of the logarithm of the relative entropy.}
	\label{1AEntropyPlot}
\end{figure}

\subsubsection{Evolution of the Total Mass}

Most other works consider the case of an unbounded domain with confining potential or no-flux boundary conditions which yields a preservation of the total mass. This is not true in our case where we have
% While most other authors are often interested in particle systems where the total mass of the particles in the domain is conserved we analyze a setting in which the total mass evolves due to boundary conditions following the equation
\begin{align*}
	\frac{d}{dt} \int_\Omega \rho \, ~ dx
	= \int_{\Gamma_{\text{in}}} \alpha ~ d\sigma 
	- \int_{\Gamma_{\text{out}}} \beta \rho \, ~ d\sigma
\end{align*}
as vesicles can enter or leave the domain at $\Gamma_{\text{in}}$ or $\Gamma_{\text{out}}$.
We want to examine this evolution numerically using two different initial conditions.
As $\rho$ converges exponentially to some equilibrium state, its mass also converges exponentially to the mass of the stationary solution. Yet in contrast to the relative entropy, the evolution of the mass is not monotone in time. To shed light on this phenomena we are now looking for initial functions which enforce non monotone mass evolution.
In Figure \ref{1AExampleforNonMonotoneMassEvolution} (a) the initial function is
\begin{align}
\begin{split}\label{eq:rho0mass1}
\rho_0(x) = 
	\begin{cases}
		1.9 &\text{if }  x < 0.5,\\
		1.9 ~ \big(0.95\cos(4 \pi x)+0.95 \big) &\text{if } 0.5 \leq x \leq 0.75,\\
		0 &\text{if }  x > 7.5.
	\end{cases}
	\end{split}
\end{align}
The mass of this initial function is about 1.1863 which is more than the mass of the equilibrium state 1.0703, so in the long run the mass evolution should be a decreasing function. 
But at the beginning the mass of $\rho$ is increasing in time. 
This can be explained by the fact that particles are pumped into the domain with rate $\alpha$ at $x=0$ whereas no particles can leave the domain as they are no particles at the exit of the domain at $x=1$.
Thus the mass increases until the drift term has transported (a substantial number of) particles to $x=1$.
In Figure \ref{1AExampleforNonMonotoneMassEvolution} (b) the initial function is
\begin{align}
 \begin{split}\label{eq:rho0mass2}
\rho_0(x) = 
	\begin{cases}
		0 &\text{if }  x < 0.9,\\
		3000(x-0.9)^2 &\text{if }  x \geq 0.9,
	\end{cases}
\end{split}
\end{align}
which yields the opposite behavior. 
At first the mass of the particles concentration is decreasing and then it increases.
This is due to the fact that the particle concentration at $x=1$ is 30 which is very much compared to the rest of the domain.
As the outflux is proportional to the concentration at $x=1$ there are more particles leaving the domain than entering it.
The mass of the initial function is 1.0711 and the mass of the equilibrium state is 1.0696.
Note that one can see that the maximum of the mass of the initial function and the mass of the stationary solution is not an upper bound for the mass of $\rho$.

\begin{figure}[H]
	\begin{center}
		\begin{tabular}{cc}
\includegraphics[width=0.40\textwidth]{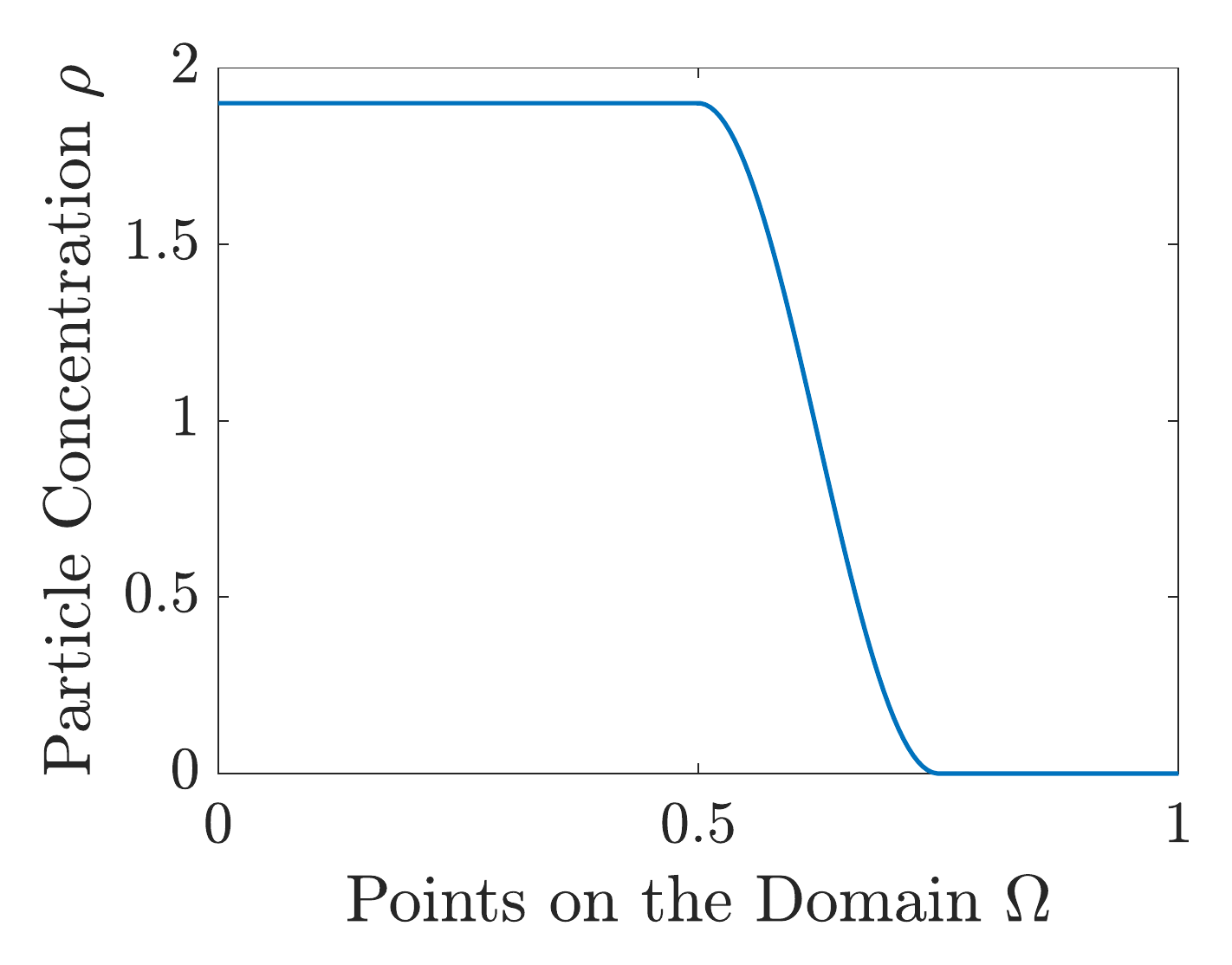} &
\includegraphics[width=0.40\textwidth]{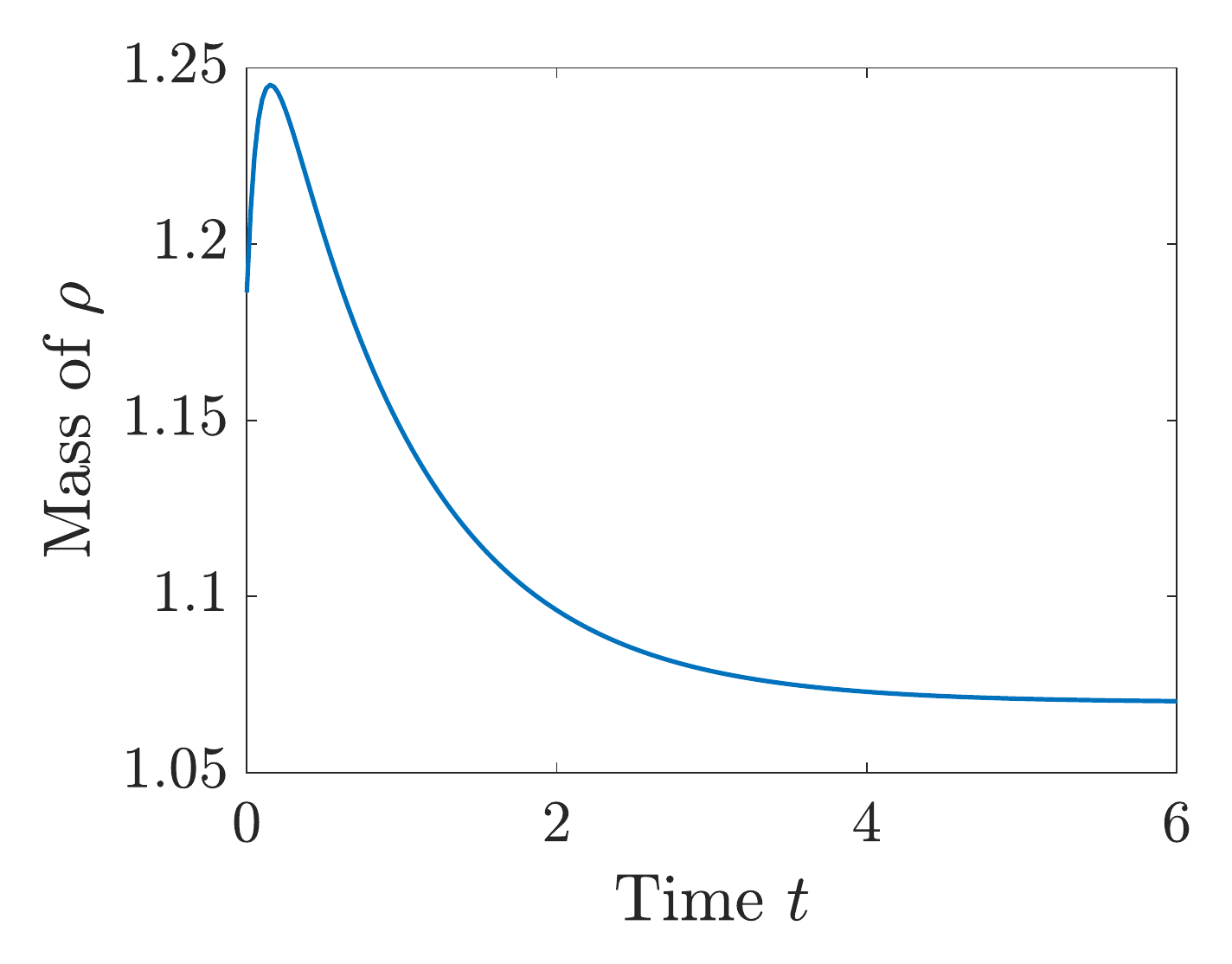} \\
		(a) $\rho_0$ as in \eqref{eq:rho0mass1} & (b) Mass evolution\\
		\includegraphics[width=0.40\textwidth]{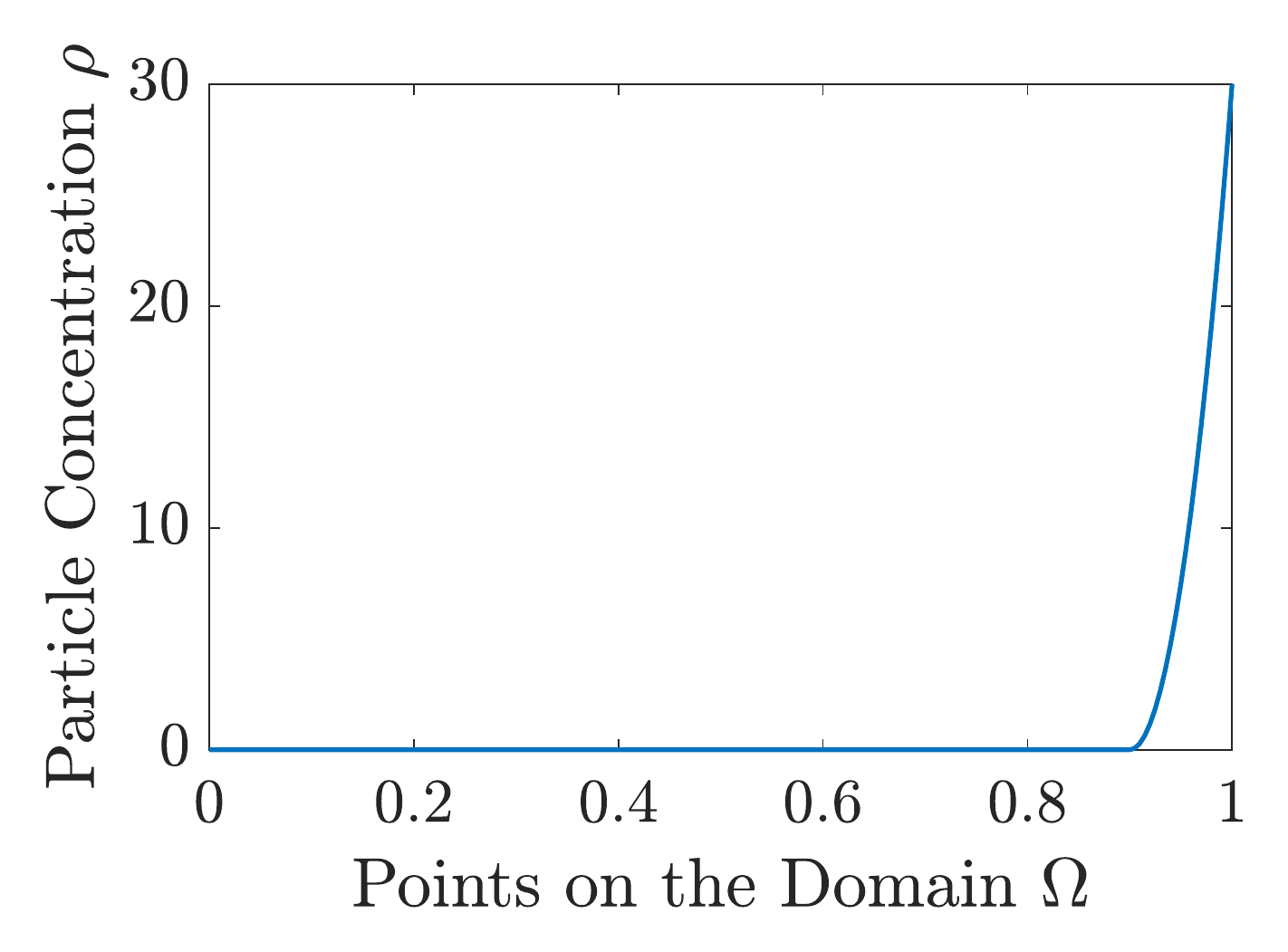} &
\includegraphics[width=0.40\textwidth]{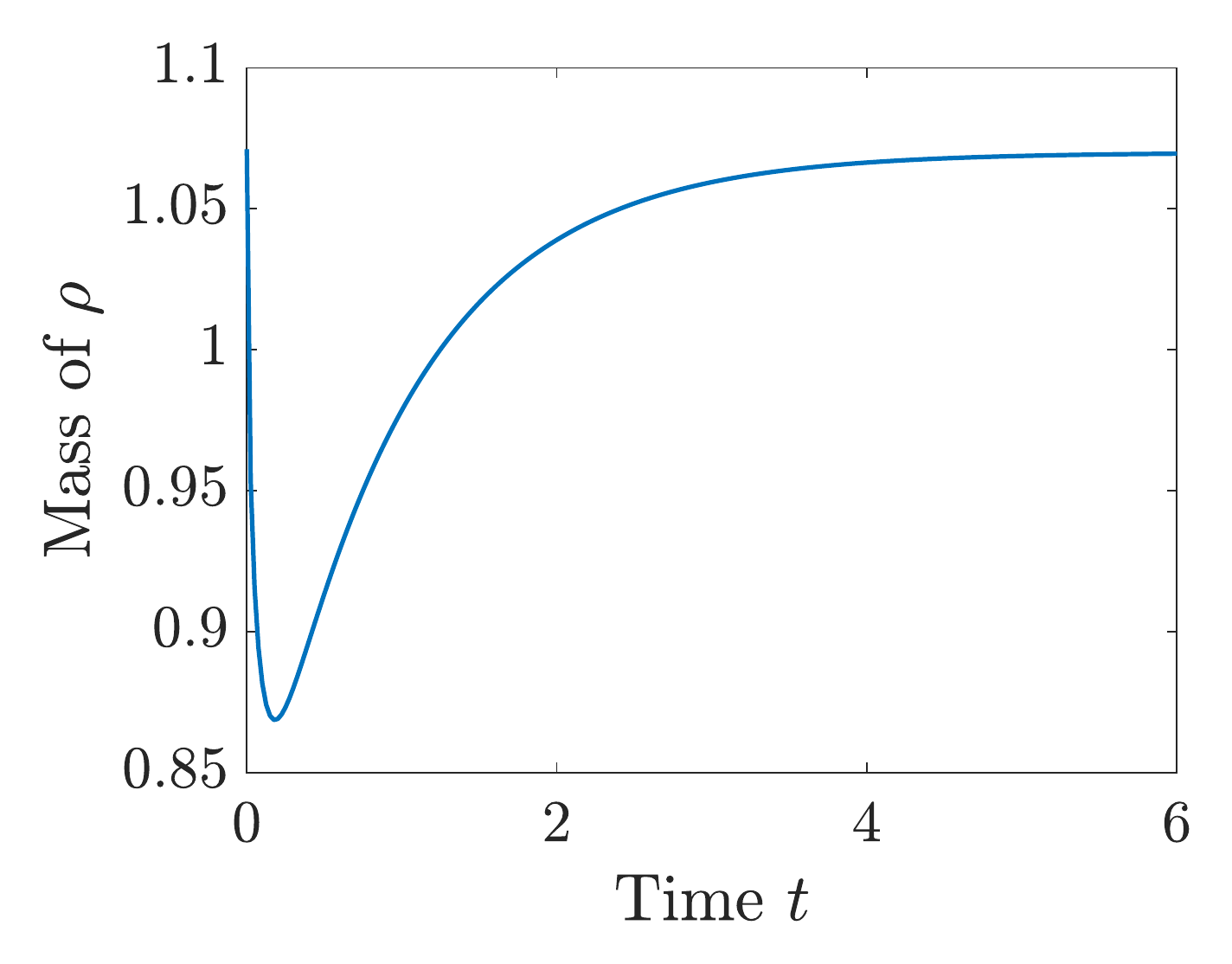} \\
		(c) $\rho_0$ as in \eqref{eq:rho0mass2} & (d) Mass evolution
	\end{tabular}
	\end{center}
	\caption{\textbf{Mass Evolution:} Two examples for non monotone mass evolution for $\alpha =1$ and  $\beta=0.9$.}
	\label{1AExampleforNonMonotoneMassEvolution}
\end{figure}

\section{Linear Model with Spatially Distributed In- and Outflux}
\label{sectionUniformSpatialInandOut}
In our second model we consider the Fokker-Planck equation where influx and outflux is modeled by reaction terms, i.e.
\begin{align}
	\partial_t \rho + \nabla \cdot (-\nabla \rho + \rho \nabla V)  = \alpha - \beta\rho  e^{-V} ,\label{Fokker-Planck1B}
\end{align}
for $t\geq 0$ and $x \in \mathbb{R}^n$, $\alpha \in L^\infty(\Omega)$ with $\alpha \geq 0, ~\beta\in L^\infty(\Omega)$ with $\beta \geq \beta_0 > 0$ for some $\beta_0 > 0$ and $V= V(x)$ is a smooth and bounded potential as in the previous section.
Furthermore we assume a \textit{no flux condition} i.e.
\begin{align}
	\label{BoundaryCondition1B}
	J \cdot n = 0 \text{ for all } x \in \partial\Omega \times (0,T).
\end{align}
% where $n$ is the normal outward and $J = - \nabla \rho + \rho \nabla V $ is the flux.
We make the following assumptions:
\begin{itemize}
\item[(B1)] The domain $\Omega \subset \mathbb{R}^n$ is a bounded with $\partial \Omega \in C^{1,1}$. 
\item[(B2)] The initial function is non-negative and fulfills $\rho_0 \in L^2(\Omega)$.
\item[(B3)] The potential $V$ is smooth and bounded, i.e. there are constants $K_2 = \inf\lbrace e^{-V}\rbrace$ and $K_3=\sup\lbrace e^{-V}\rbrace$. 
Furthermore $\nabla V$ and $\Delta V$ are also bounded and $-\Delta V - \beta e^V \leq 0$ on $\Omega$.
\item[(B4)] The functions $\alpha$ and $\beta$ satisfy $\alpha \in L^\infty(\Omega)$ with $\alpha \geq \alpha_0 > 0$ and $\beta\in L^\infty(\Omega)$ with $\beta \geq \beta_0 > 0$ for some $\alpha_0,\beta_0 > 0$.
\end{itemize}

\begin{figure}[H]
	\begin{center}
\includegraphics[width=0.3\textwidth]{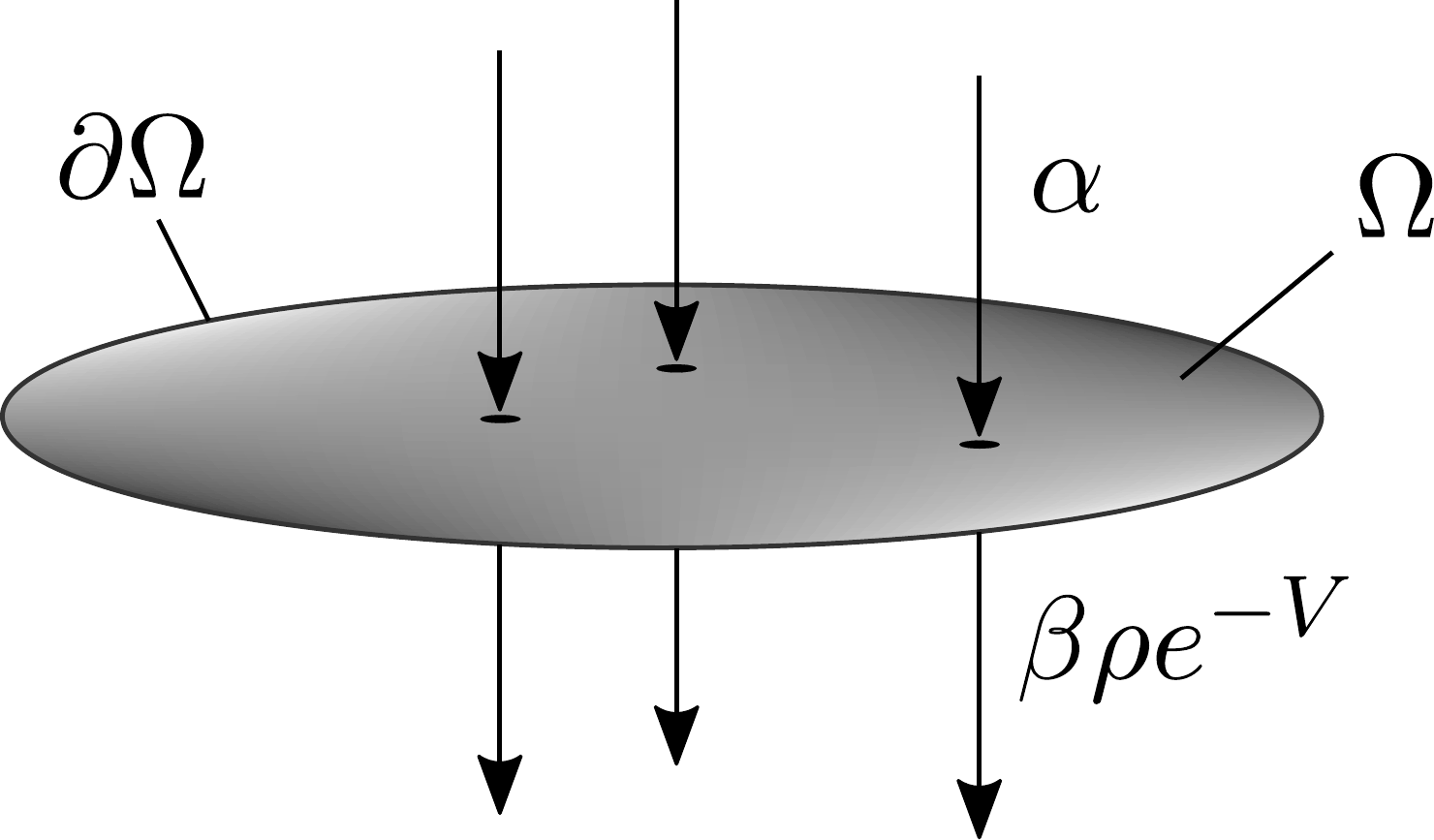} 
	\end{center}
	\caption{Sketch of the geometry of uniform spatial in- and outflux in 2D.}
\end{figure}

The reaction terms in \eqref{Fokker-Planck1B} translate to vesicles entering and leaving the domain at every point with rates $\alpha$ and $\beta\rho e^{-V}$, respectively.
% , not only on the special subset $\Gamma_{\text{in}}$ of the boundary as in the previous section. 
% The same intuition holds for the outflux term $\beta\rho  e^{-V}$. Again there is no conservation of total mass.
In this case, the stationary solution is given by $\rho_\infty = \frac{\alpha}{\beta} e^V$.
The following theorem, which we will prove in \ref{EquilibriumKapitel}, is the main result of this section:
\begin{theorem}\label{thm:ExponentialDecay1B}
	Let (B1)-(B3) hold, then equation \eqref{Fokker-Planck1B} with no flux boundary conditions obeys the following exponential decay towards equilibrium
	\begin{align*}
		\vert\vert \rho - \rho_\infty\vert\vert^2_{L^1(\Omega)} 
		\leq \frac{1}{K_4}  E(\rho_0\vert\rho_\infty)e^{-\frac{4\beta K_2}{K_1}t},
	\end{align*}
	where $K_1$ depends on $\mathrm{max} \lbrace \vert\vert \rho_\infty \vert\vert_\infty, \vert \vert \rho_0 \vert \vert_\infty \rbrace$ only and $E(\rho|\rho_\infty)$ is the logarithmic relative entropy defined in \eqref{eq:relativelogentropy}.
% 	, $\phi$ as in lemma \ref{Definitionhilfslemma} and $L$ being an upper bound for $\rho$.
	Furthermore $K_2 = \inf\lbrace e^{-V}\rbrace$ and $K_4$ come from the Czisz\'{a}r-Kullback-Pinsker inequality in Lemma \ref{CsiszarKullbacknotprobabilitymeasures}.
\end{theorem}

\subsection{Existence of the Time Dependent Problem}

\begin{definition}
	A function $\rho \in L^2(0,T;H^1(\Omega))$ with $\partial_t \rho \in L^2(0,T;H^{-1}(\Omega))$ is a \textit{weak solution} to equation \eqref{Fokker-Planck1B} supplemented with the boundary condition \eqref{BoundaryCondition1B} if the identity
	\begin{align}
		\label{1BWeakSolution}
		\int_\Omega \partial_t \rho \, \phi \, dx
		- \int_\Omega \big(-\nabla \rho + \rho \nabla V\big) \nabla \phi \, dx 
		+ \int_\Omega \beta \rho e^{-V} \phi \, dx
		= \int_\Omega \alpha \phi \, dx
	\end{align}
	holds for all $\phi \in H^1(\Omega)$ and a.~e. time $0 \leq t \leq T$.
	Again we can rewrite the weak formulation in terms of the \textit{Slotboom variable} $u\coloneqq \rho e^{-V}$ and obtain
	\begin{align}
		\label{1BWeakSolutionSchlotbom}
		\int_\Omega e^V \partial_t u \, \phi ~ dx
		+ \int_\Omega e^V\nabla u \cdot \nabla \phi ~ dx 
		+ \int_\Omega \beta ue^V \phi ~ d\sigma
		= \int_\Omega \alpha \phi ~ d\sigma.
	\end{align}
\end{definition}

\begin{lemma}
	If $\nabla V \in L^\infty (\Omega)$ and $\rho_0 \in L^2(\Omega)$, there exists a unique weak solution to equation \eqref{Fokker-Planck1B} in the sense of definition \ref{1BWeakSolution}.
\end{lemma}

\begin{proof}
	We use the Slotboom formulation, as in \eqref{1BWeakSolutionSchlotbom}, which yields the bilinear form
	\begin{align*}
		B[u, \phi]
		= \int_\Omega e^V\nabla u \cdot \nabla \phi ~ dx 
		+ \int_\Omega \beta ue^V \phi ~ d\sigma,
	\end{align*}
	that obviously fulfills \cite[6.2.2 Theorem 2]{EvansPartialDifferentialEquations} as $V$ is bounded and $\beta \geq \beta_0 > 0$.
	Next we apply the ideas stated in \cite[pp. 353 - 358]{EvansPartialDifferentialEquations}.
\end{proof}

\subsection{Stationary Solutions}

\begin{lemma}
	For $V,\, \nabla V$ and $\Delta V$ bounded and $-\Delta V - \beta e^V \leq 0$ on $\Omega$, the unique stationary solution $\rho_\infty \in H^2(\Omega)$ to 
	\begin{align*}
	 \nabla \cdot (-\nabla \rho + \rho \nabla V) + \beta\rho  e^{-V} = \alpha
	\end{align*}
 with boundary conditions \eqref{BoundaryCondition1B} is given by
 \begin{align*}
  \rho_\infty = \frac{\alpha}{\beta} e^{V}.
 \end{align*}
 \end{lemma}
\begin{proof}
Rewriting the problem as 
\begin{align*}
 -\Delta \rho + \nabla\rho \cdot \nabla V + (\Delta V + \beta e^{-V})\rho = \alpha,
\end{align*}
we see that our assumption on $V$ guarantees a sign on the lower order term and the assertion follows, e.g. from \cite[Thm. 2.4.2.6]{Grisvard85}.
\end{proof}

\begin{remark}
	If we choose $\beta=\alpha=0$, we lose uniqueness of the solution of \eqref{Fokker-Planck1B}.
	But if we demand the solution to be positive, we can restore uniqueness see \cite[Theorem 1.1]{DroniouJeromeVaszquezDiffusionEquations}.	
\end{remark}

Having proven the uniqueness of the stationary solution, enables us to give an upper bound for $\rho$, which we need for the equilibrium property. i.e:

\begin{lemma}
	\label{Grenzefuer1B}
 	If \eqref{Fokker-Planck1B} and \eqref{BoundaryCondition1B} hold, $\rho\geq 0$ and there is an upper bound $L$ for $\rho$ i.e. for every $t \in (0,T)$ there holds
 	\begin{align*}
 		0 \leq \vert \vert \rho \vert \vert_\infty
 		\leq \mathrm{max} \lbrace \vert\vert \rho_\infty \vert\vert_\infty, \vert \vert \rho_0 \vert \vert_\infty \rbrace 
 		\coloneqq L.
	\end{align*} 		
\end{lemma}

\begin{proof}
	We define 
	$\tilde{L} \coloneqq  \mathrm{max} \lbrace \alpha \beta^{-1}, \vert \vert \rho_0 e^{-V}\vert \vert_\infty \rbrace$, which shall be an upper bound for $u$ reminding ourselves that $ \rho_0 e^{-V} = u_0$.
	As $e^{-V} \partial_t u = \partial_t \rho$ and $\tilde{L}$ is a constant we can write
	\begin{align*}
		e^{-V}\partial_t (u-\tilde{L}) + \nabla \cdot (-e^{-V}\nabla (u-\tilde{L})) + \beta u = \alpha.
	\end{align*}
	Using the weak formulation of this equation with test function $(u-\tilde{L})_+ \in H^1(\Omega)$, which has the derivative
	  \[
     \nabla(u -\tilde{L})_+=\left\{\begin{array}{ll} \nabla(u -\tilde{L})_+, & (u -\tilde{L})_+ > 0, \\
         0, & \text{otherwise,}\end{array}\right. 
  \]	
	we obtain 
	\begin{align*}
		\partial_t \int_\Omega e^{-V} \frac{(u-\tilde{L})_+^2}{2} \, dx 
		+ \int_\Omega e^{-V} \vert \nabla (u-\tilde{L})_+\vert^2\, dx 
		= - \beta \int_\Omega (u- \frac{\alpha}{\beta})(u-\tilde{L})_+\, dx
	\end{align*}
	by integration by parts, using $J \cdot n = (e^{-V} \nabla u) \cdot n=0$ on $\partial \Omega$. As the second integral is positive, we can omit it to achieve
	\begin{align*}
		\partial_t \int_\Omega e^{-V} \frac{(u-\tilde{L})_+^2}{2} \,dx 
		\leq - \beta \int_\Omega (u- \frac{\alpha}{\beta})(u-\tilde{L})_+\, dx 
		\leq - \beta \int_\Omega (u-\tilde{L})_+^2 \,dx.
	\end{align*}
	Now we use $\frac{e^{-V}}{2} \geq C$ as $V$ is bounded, and Gronwall's lemma yields
	\begin{align*}
		\int_\Omega (u-\tilde{L})^2_+ dx \leq e^{- \frac{C}{\beta} t} \int (u_0 - \tilde{L})^2_+ dx = 0,
	\end{align*}
	so $u \leq \tilde{L}$ a.e. and therefore $\rho \leq \tilde{L} \sup\lbrace e^{-V}\rbrace=L$ a.e. in $\Omega$ and for every $t \in (0,T)$. Repeating the same argument with test function $u^{-}= \min \lbrace u,0\rbrace$ gives $0 \leq \rho$.
\end{proof}

\subsection{Long Time Behavior}
\label{EquilibriumKapitel}
Our proof of the exponential decay to equilibrium is closely related to ideas presented in \cite{DesvillettesFellnerExponentialdecay}.
A central tool in this paper which is used to give an upper bound on the relative entropy is the following lemma which we state for the sake of completeness:
\begin{lemma}\cite[Lemma 2.1]{DesvillettesFellnerExponentialdecay}
	\label{Definitionhilfslemma}
	The function
	\begin{align*}
		\phi(x,y)=\frac{x(\log(x)-\log(y))-(x-y)}{(\sqrt{x}-\sqrt{y})^2}
	\end{align*}
	is continuous on $(0, + \infty)^2$. 
	For all $x>0$ the function $\phi(x,\cdot )$ is strictly decreasing and respectively
	for all $y>0$ the function $\phi(\cdot,y)$ is strictly increasing on $(0, + \infty)$.
	Finally it satisfies
	\begin{align*}
		\lim_{x \rightarrow 0} \phi(x,y)&=1, \\ 
		\phi(y,y)&=2 \text{ and } 
		\phi(x,y) \sim \log x \text{ for } x \rightarrow \infty.
	\end{align*}
\end{lemma}
With this at hand, we proceed to the proof of theorem \ref{thm:ExponentialDecay1B}:
\vspace*{2ex}
\\
\textit{Proof of theorem \ref{thm:ExponentialDecay1B}}: The previous lemma enables us to rewrite the logarithmic entropy \eqref{eq:relativelogentropy} as
\begin{align}
	\label{1BRelativeEntropy}
	E(\rho\vert\rho_\infty) 
	&= \int_\Omega \phi ( \rho , \rho_\infty )  (P-P_\infty )^2	dx,
\intertext{where $P_\infty \coloneqq \sqrt{\rho_\infty}$ and $P \coloneqq \sqrt{\rho}$. As $\phi$ is monoton increasing in the first  and monoton decreasing in the second component, $\rho \leq L$ almost everywhere and $\rho_\infty = \frac{\alpha}{\beta}e^V \geq 0$, we gain}
	E(\rho\vert\rho_\infty) &\leq \int_\Omega \phi(L, 0)(P-P_\infty )^2 = K_1 \vert \vert P - P_\infty \vert\vert^2_2,
	\notag
\end{align}
where $K_1 \coloneqq \max \lbrace 1 , \phi(L, 0)\rbrace >0$. 
Note that we just take the maximum with $1$ to ensure that $K_1$ is positive which will make further computations more easy.
% Instead of 1 we could have taken any other positive real number.
Furthermore the entropy dissipation $D(\rho\vert\rho_\infty)$ is given by
\begin{align*}
	D(\rho\vert\rho_\infty) 
	= - \int_\Omega \partial_t \rho \log\Big(\frac{\rho}{\rho_\infty}\Big) dx
	&= -\int_\Omega ( - \nabla \cdot J + \alpha - \beta\rho e^{-V}) \log\Big(\frac{\rho}{\rho_\infty} \Big) dx
\end{align*}
and by integration by parts with no flux boundary conditions we get
\begin{align*}
	&= \int_\Omega \frac{\vert J \vert^2}{\rho} - ( \alpha - \beta\rho e^{-V}) \log\Big(\frac{\rho}{\rho_\infty} \Big) dx 
	\geq  \beta \int_\Omega e^{-V} (\rho - \rho_\infty) \log\Big(\frac{\rho}{\rho_\infty} \Big) dx \\
	&\geq  4 \beta \int_\Omega e^{-V} (P-P_\infty)^2 dx 
	\geq  4 \beta K_2 \vert \vert P - P_\infty \vert\vert^2_2,
\end{align*}
where the penultimate inequality holds because of the elementary inequality $(a-b)(\log(a) - \log(b)) \geq 4(\sqrt{a}-\sqrt{b})^2$.
As $V$ is bounded, there exists a constant $K_2> 0$ with $\inf \lbrace e^{-V}\rbrace =K_2$.
Combining the estimates for the relative entropy and the entropy dissipation, we achieve
\begin{align}
	\label{LogSob1B}
	\frac{1}{K_1} E(\rho\vert\rho_\infty) \leq \vert \vert P - P_\infty \vert\vert^2_2 &\leq \frac{1}{4\beta K_2} D(\rho\vert\rho_\infty).
\end{align}
Combining this with Gronwall's lemma, we obtain
\begin{align}
	\label{Gronwall1B}
	E(\rho\vert\rho_\infty)\leq E(\rho_0\vert\rho_\infty)e^{-\frac{4\beta K_2}{K_1}t}.
\end{align}
To conclude the proof we use the following lemma which is a generalization of the Cziszar-Kullback-Pinsker inequality for functions which are not probability densities.  $\qquad \square$
\begin{lemma} \cite[Lemma A.1]{HaskovecHittmeirMarkowichMielkeDecay}
	\label{CsiszarKullbacknotprobabilitymeasures}
	Let $\Omega$ be a measurable domain in $\R^d$.
	Let $\rho,\rho_\infty \colon \Omega \rightarrow \R^+$ be measurable.
	Then,
	\begin{align}
		E(\rho\vert\rho_\infty) \geq \underbrace{\frac{3}{2\vert\vert \rho \vert\vert_{L^1(\Omega)} + 4\vert\vert \rho_\infty \vert\vert_{L^1(\Omega)}}}_{K_4} \vert\vert \rho-\rho_\infty \vert\vert^2_{L^1(\Omega)}.
		\label{equationCsiszarKullbacknotprobabilitymeasures}
	\end{align}
\end{lemma}
Because of lemma \ref{Grenzefuer1B} the constant $K_4$ in lemma \ref{CsiszarKullbacknotprobabilitymeasures} is finite.
Combining inequality \eqref{Gronwall1B} with \eqref{equationCsiszarKullbacknotprobabilitymeasures} we conclude the desired exponential convergence, i.e.,
\begin{align*}
	\vert\vert \rho - \rho_\infty\vert\vert^2_{L^1(\Omega)} \leq \frac{1}{K_4} E(\rho_0\vert\rho_\infty)e^{-\frac{4\beta K_2}{K_1}t}.
\end{align*}

\begin{remark}
Combining the results of this section with the dissipation of the logarithmic entropy in section 3.3 it seems natural that one can also conclude the exponential convergence to equilibrium in the case of the Fokker-Planck equation with both in- and outflow in the bulk and on the boundary, i.e. 
\begin{align}
	\partial_t \rho + \nabla \cdot (-\nabla \rho + \rho \nabla V)  = \tilde \alpha - \tilde \beta\rho  e^{-V} ,
\end{align}
 with boundary conditions \eqref{BoundaryCondition1Ain}-\eqref{BoundaryCondition1Aisolated}. However, this is not possible with the current techniques for reaction-diffusion equations that we used above, since they require that the stationary solution is of the form $\rho = c e^{-V}$, which is not the case with the flow boundary conditions unless $\alpha = \beta = 0$. 
\end{remark} 

% \begin{remark}
% 	Note that you cannot add a constant on $V$ without loss of generality in contrast to the standard Fokker-Planck equation because of the $e^{-V}$-factor in the outflux summand of \eqref{Fokker-Planck1B}.
% \end{remark}
% 

\begin{remark}
	Replacing $\nabla V$ by a vector field $\vec{u}$ that is not necessarily the gradient of some vector field $V$ would be an interesting generalization of this model.
\end{remark}

\subsection{Numerical Solution}
The following results are again based on a finite difference scheme with explicit time discretization, see section \ref{sec:Numerik1A} for details.

\subsubsection{Time Evolution of the Density}
In Figure \ref{1BResultNumericalSolution} one can see the time evolution of the concentration $\rho$ solving \eqref{Fokker-Planck1B} in comparison to the calculated stationary solution $\rho_\infty = \frac{\alpha}{\beta}e^{V(x)}$ for $\alpha =1,~ \beta=0.9$ and initial concentration $\rho_0(x) = -0.1x +1.2$.
In contrast to the previous section there is no flow direction determined by the in- and outflux terms as particles enter the domain uniform in space.
\begin{figure}[h]
	\centering
    \includegraphics[width=0.40\textwidth]{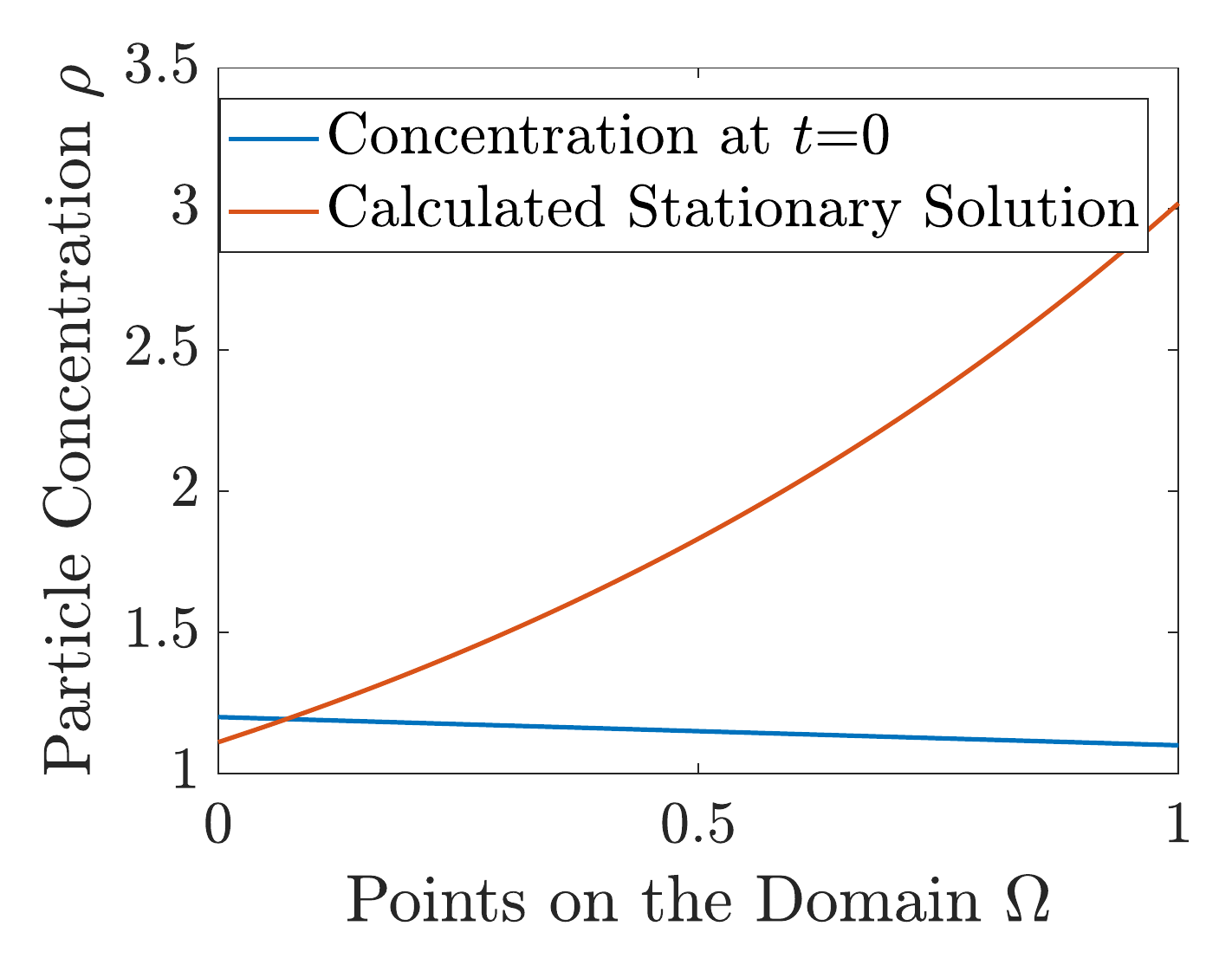}
    \includegraphics[width=0.40\textwidth]{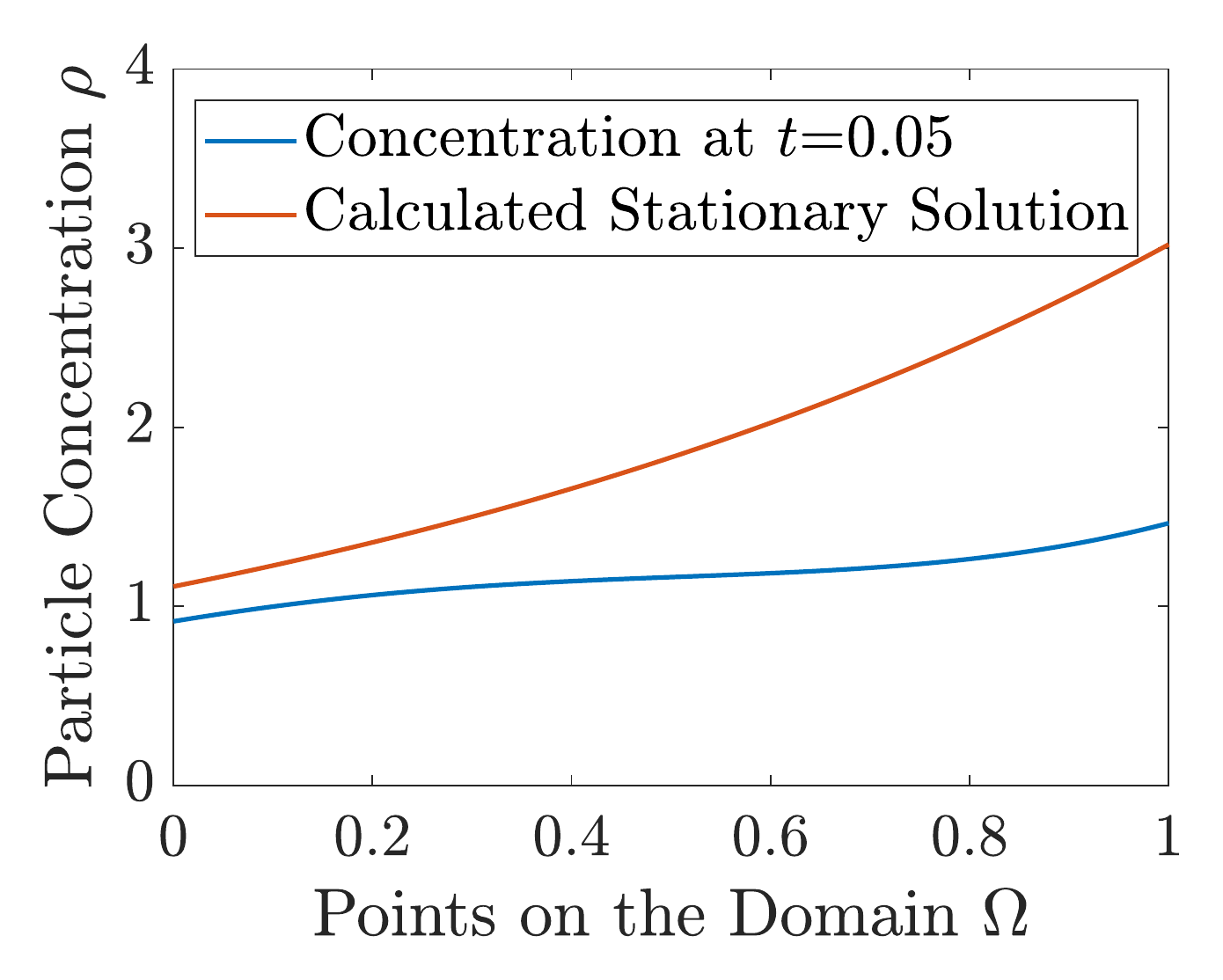}\\
    \includegraphics[width=0.40\textwidth]{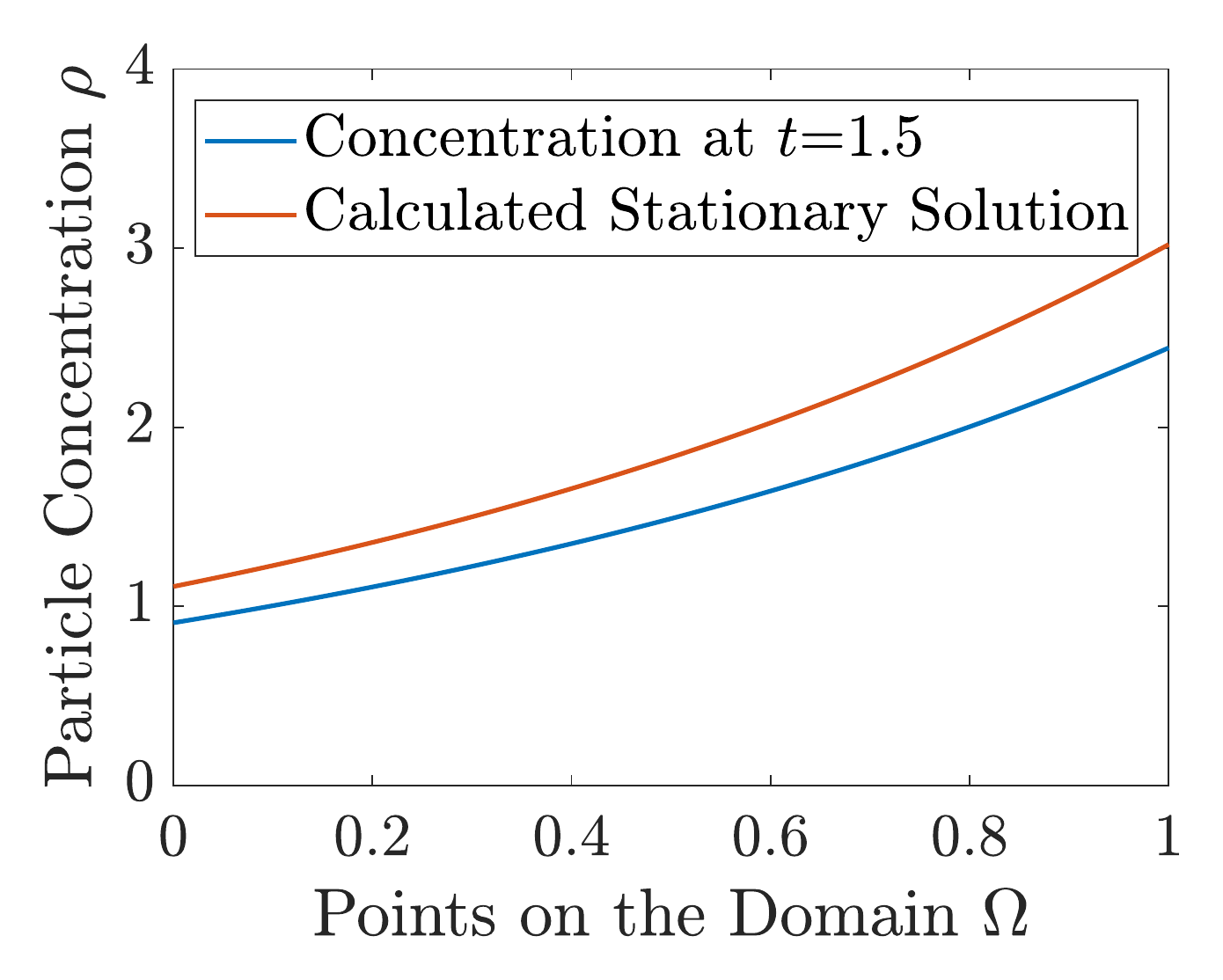}
    \includegraphics[width=0.40\textwidth]{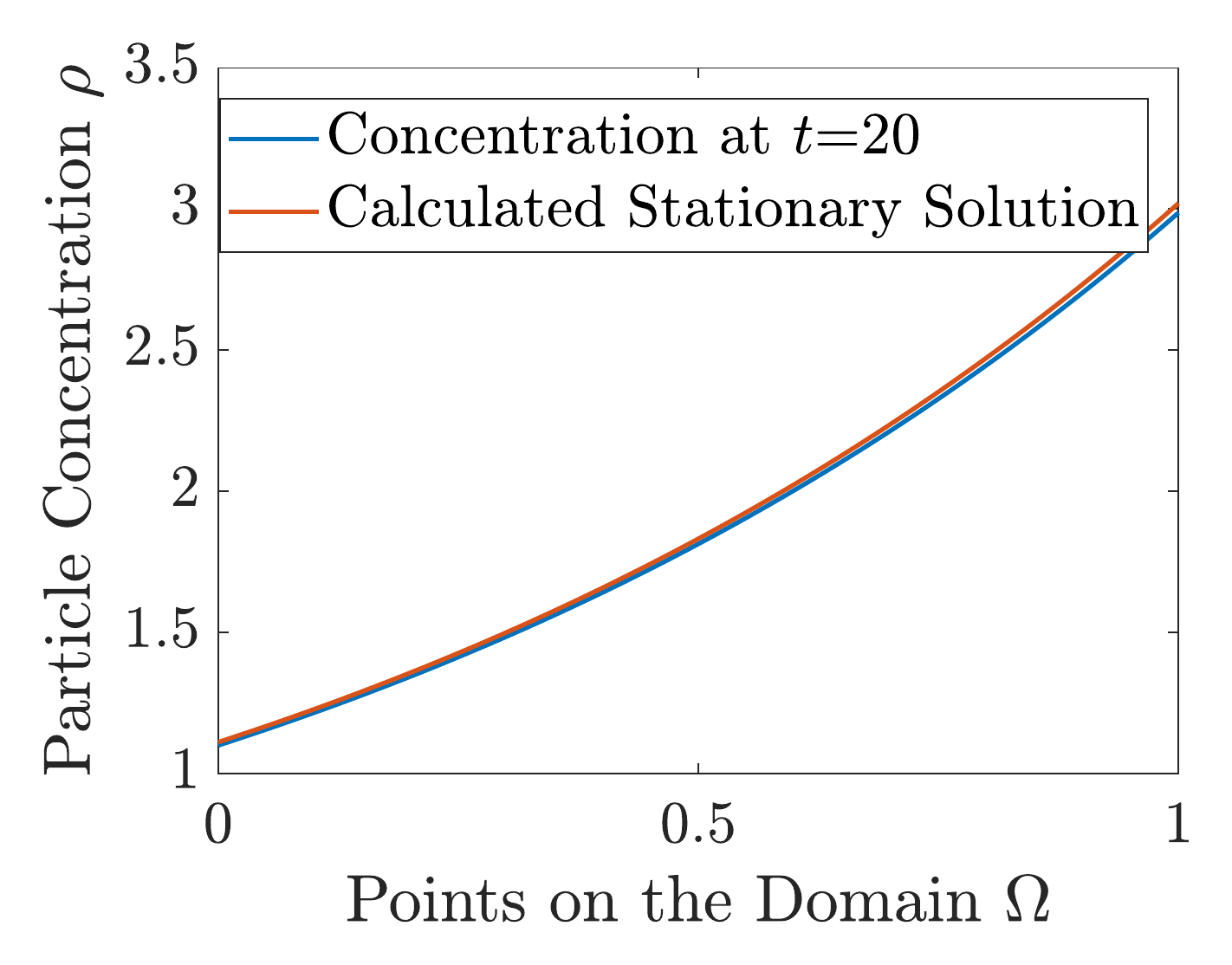}
	\caption{\textbf{Evolution over Time of the Particle Concentration:} Solution of \eqref{Fokker-Planck1B} in comparison to the calculated stationary solution $\frac{\alpha}{\beta}e^{V(x)}$ for $\alpha =1, \beta=0.9$, initial particle concentration $\rho(x) = -0.1x+1.2$. (a) The initial particle concentration at $t=0$, (b) + (c) particle concentration at $t=0.05$ and $t=1.5$, (d) equilibrium state at $t=20$.}
	\label{1BResultNumericalSolution}
\end{figure}

\subsubsection{Convergence Rates for the Relative Entropy}
The numerical solution of the logarithm of the relative entropy \eqref{1BRelativeEntropy} for the one dimensional version of \eqref{Fokker-Planck1B} for $\alpha =1, ~\beta=0.9$ and initial concentration $\rho(x) = -0.1x+1.2$ can be seen in Figure \eqref{1BPlotRelativeEntropyfunction}.
For the data given above the relative entropy has roughly the shape of $0,22e^{-1.04t}$ in the interval $t \in [0,15]$.
After that machine precision is reached.
Here we did not use the explicit stationary solution $\rho_\infty = \alpha/\beta e^x$ as the reference value in the relative entropy but we calculated the stationary solution numerically up to machine precision.
% accuracy of $10^{-16}$.

\begin{figure}[t]
	\centering
		\includegraphics[width=0.40\textwidth]{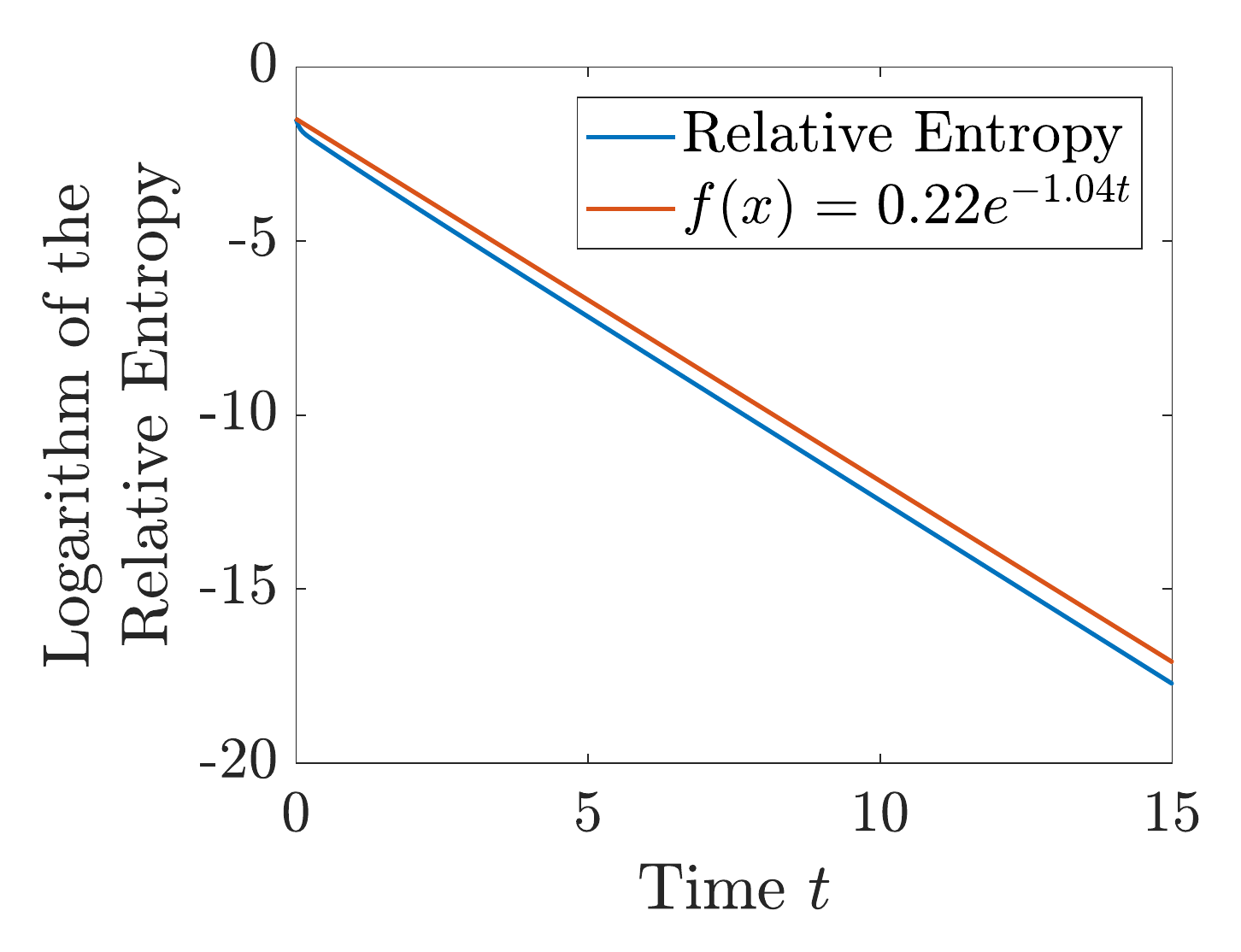} 
	\caption{\textbf{Relative Entropy:} Natural logarithm of the relative entropy \eqref{1BRelativeEntropy} for the one dimensional version of \eqref{Fokker-Planck1B} for $\alpha =1, \beta=0.9$, initial concentration $\rho(x) = 0.1x +1$.}
	\label{1BPlotRelativeEntropyfunction}
\end{figure}

\section{Nonlinear Model with Spatially Distributed In- and Outflux}
\label{sectionUniformStatialInandOutConstraint}

As vesicles have a positive volume, there naturally exists a maximal number that can fit into an axon. This motivates the next generalization of the Fokker-Planck equation given as 
\begin{align}
	\label{Fokker-Planck1C}
	\partial_t \rho + \nabla \cdot (-\nabla \rho + \rho(1-\rho) \nabla V) = \alpha ( 1- \rho) - \beta\rho e^{-V} \text{ on } \Omega \times (0,T)
\end{align}
with the same properties for $\rho, \alpha$ and $\beta$ as in the previous sections. The additional nonlinear term $\rho(1-\rho)$ and the modification of the inflow rate by $(1-\rho)$ ensure that the density stays within $[0,1]$ for all times, where $1$ corresponds to the scaled maximal density of vesicles.
Again we assume no flux boundary conditions $J \cdot n = 0$ on $\partial \Omega \times (0,T)$ and the following three assumptions:
\begin{itemize}
\item[(C1)] The domain $\Omega \subset \mathbb{R}^n$ is a connected and bounded with $\partial \Omega \in C^{1,1}$.
\item[(C2)] The initial condition satisfies $\rho_0 \in W^{2-2/p,p}(\Omega)$ for some fixed $2<p<3$ and the box constraints $0 < \rho_0 < 1$.
\item[(C3)] The potential $V$ is smooth and bounded with $\nabla V \in L^\infty (\Omega)$.
\end{itemize}
Again we want to show exponential decay to equilibrium which is surprisingly simple although in contrast to the previous settings \eqref{Fokker-Planck1C} is not a linear problem anymore. We chose the logarithmic entropy 
\begin{align*}
	E(\rho)&= \int_\Omega h(\rho)~dx \quad\text{ with } \quad h(\rho) = \rho \log(\rho) + (1-\rho)\log(1-\rho)+\rho V.
\end{align*}
for equation \eqref{Fokker-Planck1C} where $h(\rho)$ denotes the entropy density. We obtain the corresponding relative entropy
\begin{align}
	E(\rho|\rho_\infty) &= \int_\Omega \rho \log\Big(\frac{\rho}{\rho_\infty}\Big)
	+ (1-\rho) \log\Big( \frac{1-\rho}{1-\rho_\infty} \Big) ~ dx
	\label{1CRelativeEntropy}
\end{align}
motivated by viewing $\rho$ and $1-\rho$ as two different types of species. Our aim is the following result, which we will prove in section \ref{1CEquilibrium}:
\begin{theorem}\label{thm:exponential1C}
	Let (C1)-(C3) and (B4) hold. Then every weak solution to equation \eqref{Fokker-Planck1C} with no flux boundary conditions obeys the following exponential decay towards equilibrium:
	\begin{align*}
		\vert\vert \rho - \rho_\infty\vert\vert^2_{L^1(\Omega)} 
		\leq 2 E(\rho_0\vert\rho_\infty)e^{-\tilde{C}t},
	\end{align*}
	with $\tilde{C} = \alpha \min \lbrace 1 , \inf \frac{1-\rho_\infty}{\rho_\infty} \rbrace$.
\end{theorem}

\subsection{Existence of the Time Dependent Problem}

\begin{definition}
	We say that a function $\rho \in L^2(0,T;H^1(\Omega))$ with $\partial_t \rho \in L^2(0,T;H^{-1}(\Omega))$ is a \textit{weak solution} to equation \eqref{Fokker-Planck1C}, supplemented with the boundary condition $J\cdot n=0$ on $\partial\Omega$, if the identity
	\begin{align}
		\label{1CWeakSolution}
		\int_\Omega \partial_t \rho \, \phi \, dx
		- \int_\Omega \big(-\nabla \rho + \rho(1-\rho) \nabla V\big)\cdot \nabla \phi dx 
		+ \int_\Omega \beta \rho e^V \phi dx
		- \int_\Omega \alpha (1-\rho) \phi dx
		= 0
	\end{align}
	holds for all $\phi \in H^1(\Omega)$ and a.e. $0 \leq t \leq T$.
\end{definition}

\begin{lemma}
	Let (C1)-(C3) and (B4) hold, then for every $\rho_0 \in L^1(\Omega)$ with $0 \le \rho_0 \le 1$, there exists a weak solution to equation \eqref{Fokker-Planck1C} in the sense of \eqref{1CWeakSolution} which satisfies $0 \le \rho \le 1$ a.e. in $\Omega$.
	With 
	$$u = h'(\rho) = \log \rho - \log( 1- \rho) - V$$
	being the entropy variable, this solution fulfills the dissipation inequality
	\begin{align} \label{eq:1CExistenceEntropyDissipation}
		\begin{split}
		%\int_\Omega h(\rho) ~ dx 
		E(\rho)
		+ \int_0^T \Big(\int_\Omega \rho(1-\rho) \vert \nabla  u \vert^2 ~ dx
		-  \int_\Omega \alpha (1-\rho) u ~ dx  +  \int_\Omega \beta \rho e^{-V} u ~ dx \Big)\;ds
% 		\leq \int_\Omega h(\rho_0) ~ dx.
        \leq E(\rho_0).
		\end{split}
	\end{align}		
\end{lemma}

\begin{proof}
	This proof is based one implicit Euler discretization, following e.g. \cite{BurgerPietschmannFlowCharacteristics,GomesWolframParameterEstimation}, to which we refer for more details. We start by rewriting \eqref{Fokker-Planck1C}, using the entropy variable $u$ and exploiting its formal gradient flow structure, as 
	$$
	\partial_t \rho + \nabla \cdot (\rho(1-\rho) \nabla u) - \alpha ( 1- \rho) + \beta\rho e^{-V} = 0.
	$$
	Now fix $N \in \N$ and consider a discretization of $(0,T]$ into subintervals $(0,T] = \cup_{k=1}^{N}[(k-1)\tau, k\tau]$ with time steps $\tau = \frac{T}{N}$, we obtain the following sequence of elliptic problems
	\begin{align}
		\label{1CExistenceEulerDis}
		0 = \frac{\rho_k - \rho_{k+1}}{\tau} - \nabla \cdot (\rho_{k+1}(1-\rho_{k+1})\nabla u_{k+1}) - \alpha ( 1- \rho_{k+1}) + \beta\rho_{k+1} e^{-V}.
	\end{align}	
	The existence of a solution $\rho_{k+1}$ (given $\rho_k$) to the nonlinear equation \eqref{1CExistenceEulerDis} can be proven via a fixed point argument, see \cite[Theorem 3.5]{BurgerPietschmannFlowCharacteristics}. In particular, using the transformation $\rho_{k+1} = h'^{-1}(u_{k+1})$ enforces the bounds $0 \le \rho_{k+1} \le 1$ (sometimes called boundedness by entropy).

	% 	Note that this proof uses a one-to-one transformation into entropy variables guaranteeing $0\leq \rho \leq 1$.
	In order to be able to pass to the limit $\tau \to 0$, we use the discrete entropy dissipation to get a priori bounds. As the entropy density $h(\rho) = \rho \log \rho - (1-\rho) \log(1-\rho) - \rho V$ 	is strictly convex for $\rho \in S^0$, with $S^0$ being the interior of $S = \lbrace \rho \in \R ~ \vert ~ 0 \leq \rho \leq 1\rbrace$ we obtain
	\begin{align*}
		h(\rho_k) - h(\rho_{k-1}) \leq h'(\rho_k) ( \rho_k - \rho_{k-1}).
	\end{align*}
	Now taking $\phi = u_k \in H^1(\Omega)$ as test function in \eqref{1CExistenceEulerDis}, and using $\rho_k = h'^{-1}(u_k)$, we obtain the discrete entropy dissipation given as 
	\begin{align} \label{eq:1CExistenceRecursion}
		\begin{split}
		\int_\Omega h(\rho_k) ~ dx 
		+ \tau \int_\Omega \rho_k(1-\rho_k) \vert \nabla  u_k \vert^2 ~ dx
		-  \tau \int_\Omega \alpha (1-\rho_k) u_k ~ dx  \\
		+  \tau \int_\Omega \beta \rho_k e^{-V} u_k ~ dx
		\leq \int_\Omega h(\rho_{k-1}) ~ dx.
		\end{split}
	\end{align}
	Solving the recursion then yields
	\begin{align} \label{eq:1CExistenceRecursionDissolved}
		\begin{split}
		\int_\Omega h(\rho_k) ~ dx 
		+ \tau \sum_{j=1}^{k} \Big(\int_\Omega \rho_k(1-\rho_k) \vert \nabla  u_k \vert^2 ~ dx
		-  \int_\Omega \alpha (1-\rho_k) u_k ~ dx  \\
		+  \int_\Omega \beta \rho_k e^{-V} u_k ~ dx \Big)
		\leq \int_\Omega h(\rho_0) ~ dx.
		\end{split}
	\end{align}	
	To pass to the limit $\tau \rightarrow 0$ we denote by $\rho_k$ a sequence of solutions to \eqref{eq:1CExistenceRecursion}. 	We define $\rho_\tau(x,t) = \rho_k(x)$ for $x \in \Omega$ and $t \in ((k-1)\tau, k \tau)]$. 
	Then for $\tau \leq t \leq T$, the function $\rho_\tau$ solves the following problem
	\begin{align} \label{eq:1CExistenceInequality}
		\begin{split}
		&\int_0^T \int_\Omega \Big( \frac{1}{\tau} (\rho_\tau - \sigma\tau \rho_\tau)\phi + \rho_\tau(1-\rho_\tau) \nabla u_\tau \cdot \nabla \phi ~ dx  \\
		&\qquad -  \int_0^T \int_\Omega \alpha (1-\rho_\tau)  \phi ~ dx
		+  \int_0^T \int_\Omega \beta \rho_\tau \phi ~dx = 0,
		\end{split}	
	\end{align}
	where $\sigma_\tau$ denotes the shift operator, that is $(\sigma_\tau \rho_\tau)(x,t) = \rho_\tau (x, t- \tau)$ and for all test functions $\phi \in L^2(0,T;H^1(\Omega))$. Next the entropy dissipation inequality \eqref{eq:1CExistenceRecursionDissolved} becomes
	\begin{align}\label{eq:EntropyDissipationPWCont}
		\begin{split}
		\int_\Omega h(\rho_\tau(T)) ~ dx 
		+ \int_0^T\Big(\int_\Omega \rho_\tau(1-\rho_\tau) \vert \nabla  u_\tau \vert^2 ~ dx
		-  \int_\Omega \alpha (1-\rho_\tau) u_\tau ~ dx  \\
		+  \int_\Omega \beta \rho_\tau e^{-V} u_\tau ~ dx \Big)
		\leq \int_\Omega h(\rho_0) ~ dx.
		\end{split}
	\end{align}
	Following \cite[Appendix, Lemma 1]{GomesWolframParameterEstimation}, 
% 	let $\rho \in L^2(\Omega)$ and $\rho \in S^0$ a.e. be such that $u = h'(\rho) \in H^1(\Omega)$, then 
	there exists a constant $C$ such that
	\begin{align*}
		\int_\Omega \rho(1-\rho) \vert \nabla  u_k \vert^2 ~ dx 
		- \int_\Omega \alpha (1-\rho) u ~ dx
		+ \int_\Omega \beta \rho e^{-V} u ~ dx
		\geq 2 \int_\Omega \vert \nabla \rho\vert^2 ~ dx - C,
	\end{align*}
	which gives the a-priori estimate $\Vert \rho_\tau \Vert_{L^2(0,T;H^1(\Omega))} \leq K$ when combined with \eqref{eq:EntropyDissipationPWCont}. Thus, upon extraction of a subsequence, $\rho_\tau$ converges strongly in $L^2(\Omega)$.
Together with the weak convergence of $\nabla u_\tau$, this is enough to pass to the limit in \eqref{eq:1CExistenceRecursionDissolved} in all terms but the first one. There we have to apply a special version of the Aubin-Lions lemma for piece-wise constant interpolations \cite[Thm 1]{Dreher20123072} which allows us to take $\tau \to 0$. Finally, taking the limit in \eqref{eq:EntropyDissipationPWCont} yields the desired entropy dissipation inequality.
\end{proof}

\subsection{Stationary Solution}

\begin{lemma}
	There exists exactly one stationary solution $\rho_\infty \in H^1(\Omega)$ of equation \eqref{Fokker-Planck1C} with no flux boundary conditions given by
	\begin{align}
	\rho_\infty &= \frac{\frac{\alpha}{\beta}e^{V}}{1+\frac{\alpha}{\beta}e^{V}} \in [0,1].
	\label{1CStationarySolution} 
	\end{align}
\end{lemma}
\begin{proof}
    Uniqueness of the stationary solution is a direct consequence of Theorem \ref{thm:exponential1C}. Indeed, assuming that there are two different stationary solutions $\rho_{\infty}$ and $\tilde \rho_{\infty}$, inserting them into \eqref{1CGronwall} yields
%     one can consider the first stationary solution as the initial function and the second as the actual stationary solution. 	As $\rho_{\infty_1}$ is not only the initial function but also a stationary solution, we gain with $\rho = \rho_{\infty_1}$ 
	\begin{align*}
		E(\rho_{\infty}\vert \tilde \rho_{\infty} ) \leq e^{-\tilde Ct} E(\rho_{\infty}\vert \tilde \rho_{\infty} ).
	\end{align*}
	As the left hand side of the inequality is a constant whereas the right side is a decreasing function in $t$, we obtain a contradiction for $t$ large enough.
\end{proof}
% 
% Obviously the condition $0 \leq \rho_\infty \leq 1$ is fulfilled.
% And again if one wished to include a diffusion constant $D \in \R$, one would have to replace $V$ by $V/D$ in \eqref{1CStationarySolution}. 

\subsection{Long Time Behaviour} 
\label{1CEquilibrium}

First we rewrite the reactions terms in equation \eqref{Fokker-Planck1C} as
\begin{align*}
	\alpha \big(1-\rho(1+\frac{\beta}{\alpha}e^{-V})\big) 
	= \alpha (1-\frac{\rho}{\rho_\infty}) = \alpha \big((1-\rho)-\frac{\rho}{\rho_\infty}(1-\rho_\infty)\big).
\end{align*}
Next using the analogue of \eqref{eq:1CExistenceEntropyDissipation} for the relative entropy, we see that the entropy dissipation is given by
\begin{align*}
D(\rho\vert\rho_\infty) &= \int_\Omega \rho(1-\rho)|\nabla e'(\rho|\rho_\infty)|^2 
\\
&\qquad + \alpha \big((1-\rho)-\frac{\rho}{\rho_\infty}(1-\rho_\infty)\big) \big( \log\big(\frac{\rho}{\rho_\infty}\big) - \log\big(\frac{1-\rho}{1-\rho_\infty}\big) \big)\;dx.
\end{align*}
Neglecting the first non-negative part and introducing the function
$$
e(a|b) = a\log\frac{a}{b} -a +b,
$$ 
we can further estimate the dissipation from below by 
\begin{align*}
	D(\rho\vert\rho_\infty) 
	&\geq - \int_\Omega \alpha \big((1-\rho)-\frac{\rho}{\rho_\infty}(1-\rho_\infty)\big) \big( \log\big(\frac{\rho}{\rho_\infty}\big) - \log\big(\frac{1-\rho}{1-\rho_\infty}\big) \big) ~ dx 
	\\
	&= \int_\Omega \alpha \Big(~ \frac{1-\rho}{\rho_\infty} ~ \rho_\infty \log \big( \frac{\rho_\infty}{\rho})
	+ (1- \rho) \log \big( \frac{1-\rho}{1- \rho_\infty}\big) 
	\\
	&\qquad\qquad
	+ \frac{1-\rho_\infty}{\rho_\infty} ~ \rho \log\big(\frac{\rho}{\rho_\infty}\big) 
	+ \frac{\rho}{\rho_\infty}(1-\rho_\infty)\log\big( \frac{1-\rho_\infty}{1-\rho}\big)  \Big) ~ dx 
	\\
	&= \int_\Omega \alpha \Big( \frac{1-\rho_\infty}{\rho_\infty} \big( e(\rho\vert \rho_\infty) + \rho - \rho_\infty \big)
	+ \frac{\rho}{\rho_\infty} \big( e(1-\rho_\infty \vert 1- \rho) + 1-\rho_\infty-(1-\rho)\big) 
	\\
	&\qquad\qquad +  \frac{1-\rho}{\rho_\infty} \big( e(\rho_\infty \vert \rho) + \rho_\infty - \rho \big)
	+ e(1-\rho \vert 1- \rho_\infty ) - \rho + \rho_\infty \Big) ~ dx. 
\intertext{Using the definition of the relative entropy \eqref{1CRelativeEntropy}, we obtain}
	&\geq \tilde{C} E(\rho\vert \rho_\infty) + \hat{C} E(\rho_\infty \vert \rho) 
	+ \int_\Omega \alpha ~  \frac{\rho-\rho_\infty}{\rho_\infty} ~ ( 1- \rho_\infty + \rho - 1 + \rho -\rho_\infty) ~ dx
	\\
	&= \tilde{C}  E(\rho\vert \rho_\infty) + \hat{C} E(\rho_\infty \vert \rho) 
	+ 2 \int_\Omega \alpha ~ \frac{(\rho-\rho_\infty)^2}{\rho_\infty} ~ dx
	\geq \tilde{C} E(\rho\vert \rho_\infty),
\end{align*}
where we used the nonnegativity of the relative entropy and with $\tilde{C} = \alpha \min \lbrace 1 , \inf \frac{1-\rho_\infty}{\rho_\infty} \rbrace$.
With Gronwall's lemma and the Czisz\'{a}r-Kullback-Pinsker inequality in lemma \ref{CsiszarKullbacknotprobabilitymeasures}, we finally achieve
\begin{align}
	\label{1CGronwall}
	\vert\vert \rho - \rho_\infty \vert\vert^2_{L^1(\Omega)} 
	\leq 2 E(\rho\vert\rho_\infty)
	\leq 2 E(\rho_0\vert\rho_\infty)e^{-\tilde{C}t}.
\end{align}

\subsection{Numerical Solution}

Even though we are now dealing with a nonlinear equation, we again use the fully explicit scheme of section \ref{sec:Numerik1A} and obtain the following results.

\subsubsection{Analysis of the Time Evolution}

In Figure \ref{1CResultNumericalSolution} we show the time evolution of $\rho(x,t)$ solving \eqref{Fokker-Planck1C} compared to the solution \eqref{1CStationarySolution} with $\alpha =1, \beta=0.9$, initial concentration $\rho_0(x) = -(x-0.5)^2+1$, potential $V(x)=x$ and $\Omega = [0,1]$.
We chose this particular initial particle concentration (see Figure \ref{1CResultNumericalSolution} (a)) as it has a completely different shape compared to the stationary solution and secondly has the value 1 at $x=0.5$, so that at one point of the domain the density constraint of 1 is reached. 
We did not chose the same initial function as in the two previous sections as this initial function does not fulfill the box constraint.
In Figure \ref{1CResultNumericalSolution} (b) and (c) the effect of the diffusion becomes visible as it has flatten the particle concentration and the effect of the drift effect as there are more particles at the right part of the domain than in the left part.
Finally in Figure \ref{1CResultNumericalSolution} (d) there is no difference between the stationary solution and the concentration visible.
In comparison to the results of the previous model one can see that the density constraint of 1 is never overstepped.

\begin{figure}[t]
	\begin{center}
		\begin{tabular}{cc}
\includegraphics[width=0.40\textwidth]{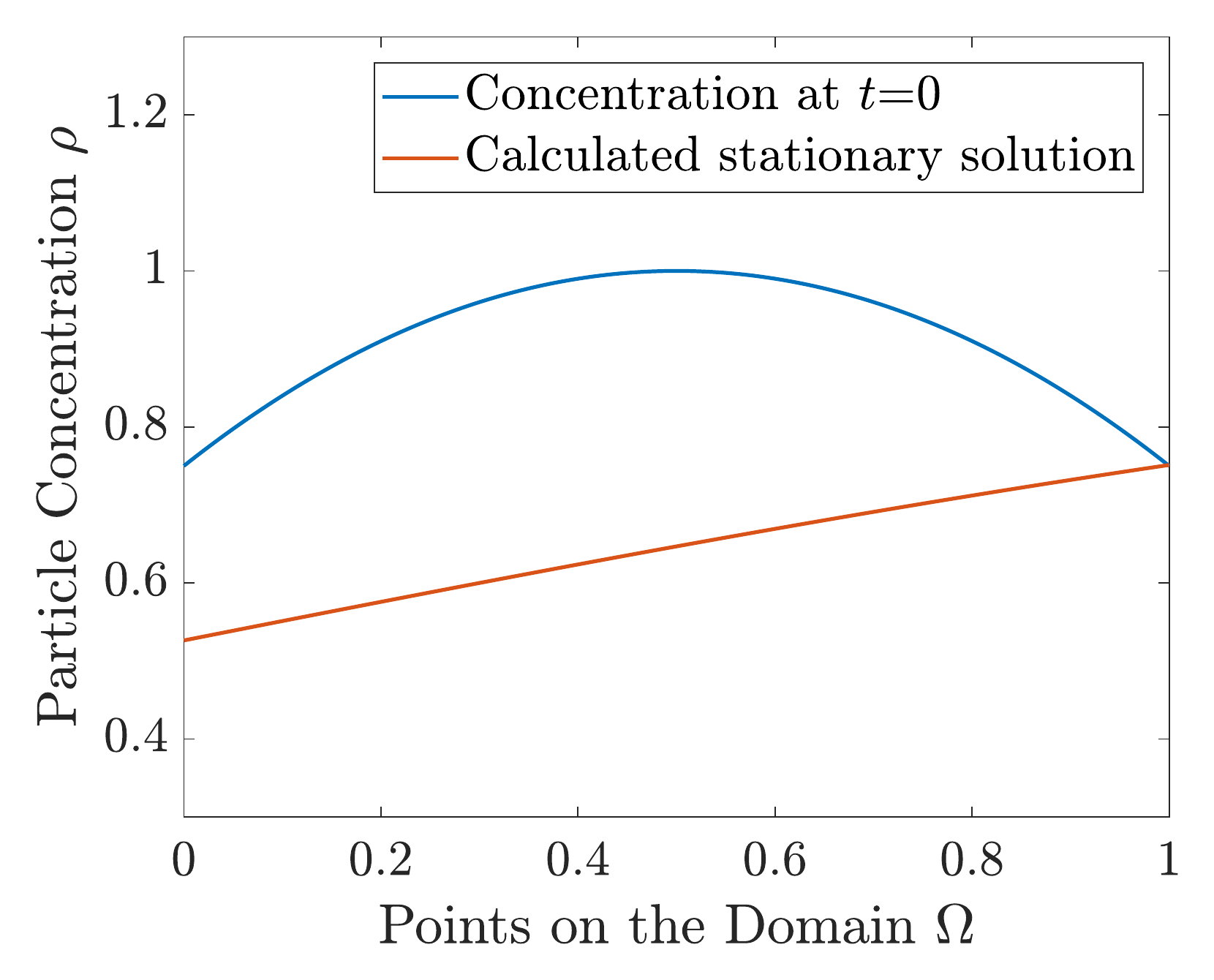} &
\includegraphics[width=0.40\textwidth]{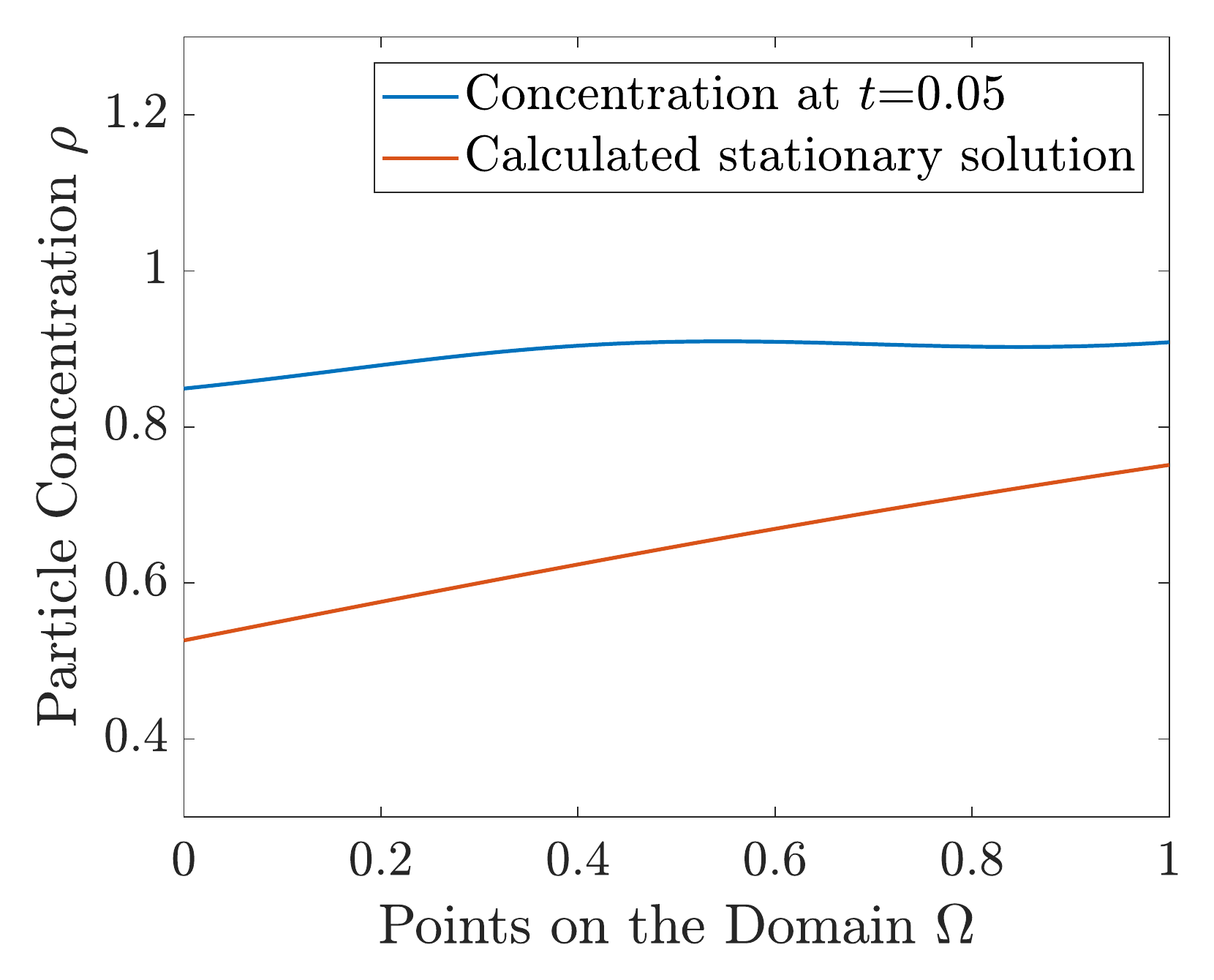}\\
			(a) & (b) \\
\includegraphics[width=0.40\textwidth]{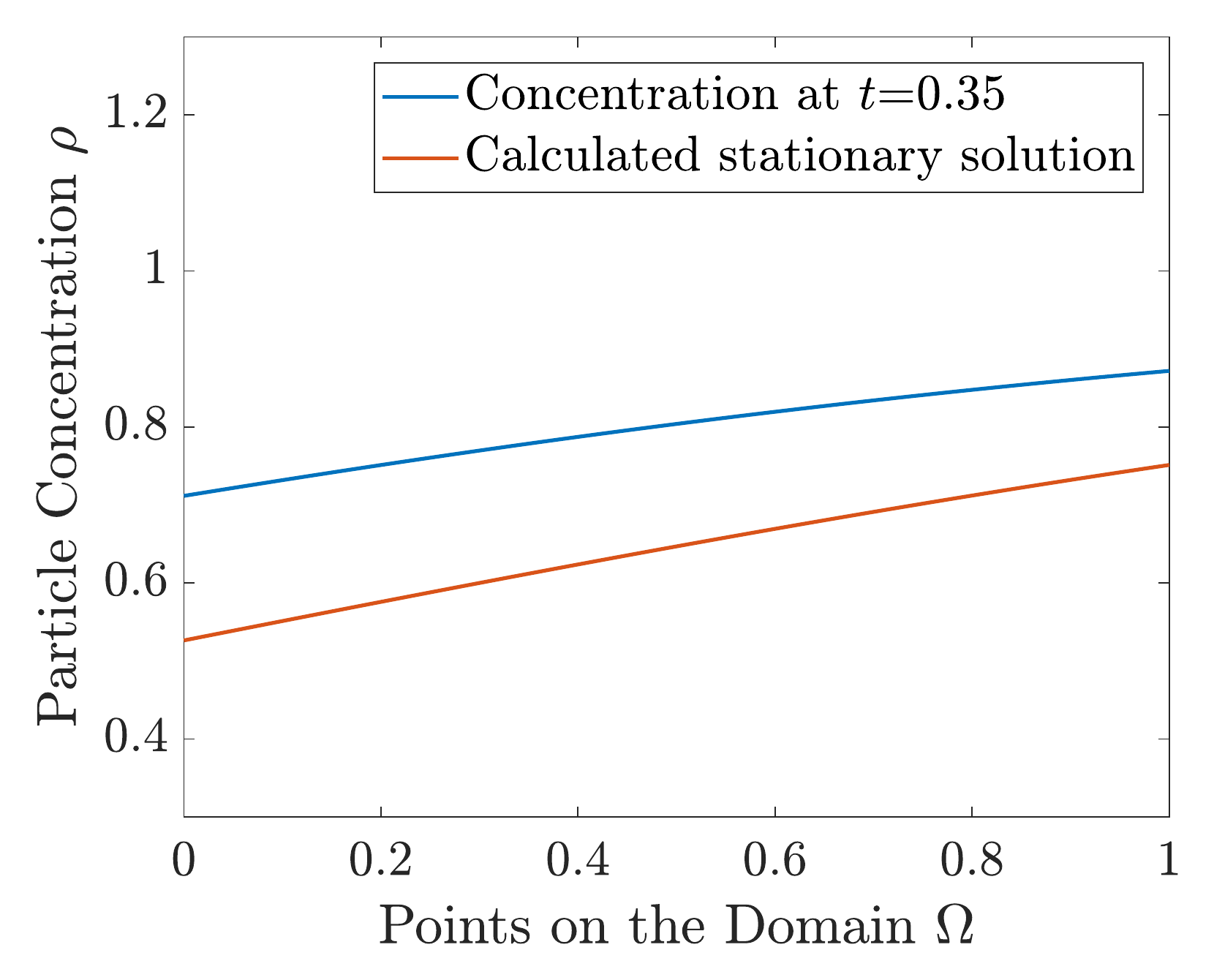} &
\includegraphics[width=0.40\textwidth]{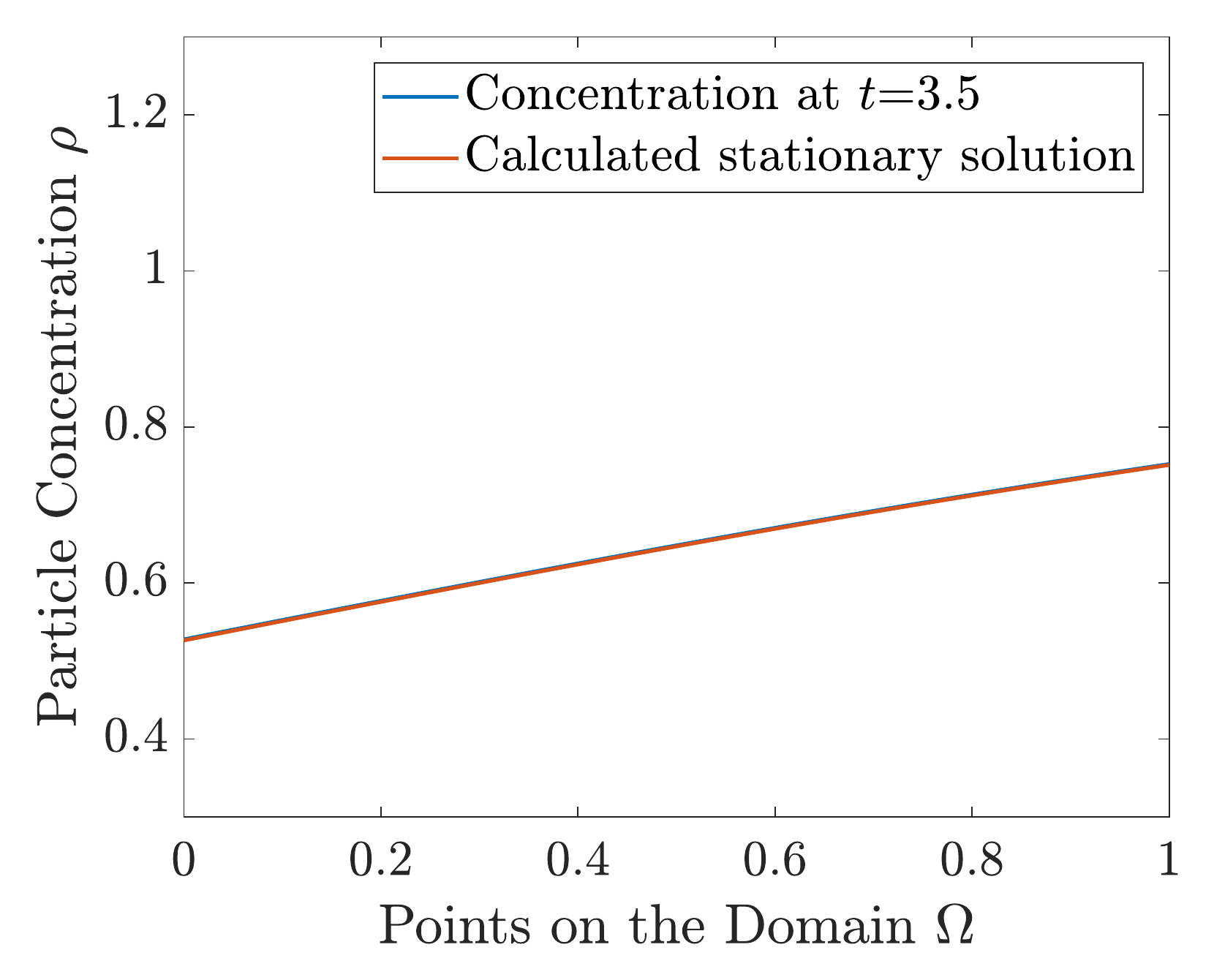} \\
			(c) & (d)		
		\end{tabular}
	\end{center}
	\caption{\textbf{Evolution over Time of the Particle Concentration:} $\rho(x,t)$ solving \eqref{Fokker-Planck1C} in comparison to the calculated solution \eqref{1CStationarySolution}  for $\alpha =1, \beta=0.9$, initial particle concentration $\rho_0(x) = -(x-0.5)^2+1$. (a) The initial concentration in comparison with the calculated stationary solution at $t=0$, (b) diffusion strongly visible $t=0.05$, (c) transport term strongly visible $t=0.35$, (d) equilibrium state at $t=3.7$.}
	\label{1CResultNumericalSolution}
\end{figure}

\subsubsection{Analysis of the Relative Entropy}

In Figure \ref{1CPlotRelativeEntropyfunction} the relative entropy for the one dimensional version of \eqref{Fokker-Planck1C} for $\alpha =1, \beta=0.9$, initial concentration $\rho_0(x) = -(x-0.5)^2+1$ and potential $V(x)=x$ can be seen.
To better compare the results with the one of the previous section we chose the same $\alpha$ and the same $\beta$.
Comparing the convergence velocity of this model with the previous one, one can see that the case without a density constraint obeys a quicker exponential decay.
This can be explained by the following intuition:
Figuratively the relative entropy measures the distance to an equilibrium state and the convergence rate how quick this status is reached.
In this setting the influx and the drift term are multiplied by the factor $(1-\rho)$ whereas in the previous section it was not, so they have less influence than in the previous model.

\begin{figure}[t]
	\begin{center}
		\includegraphics[width=0.38\textwidth]{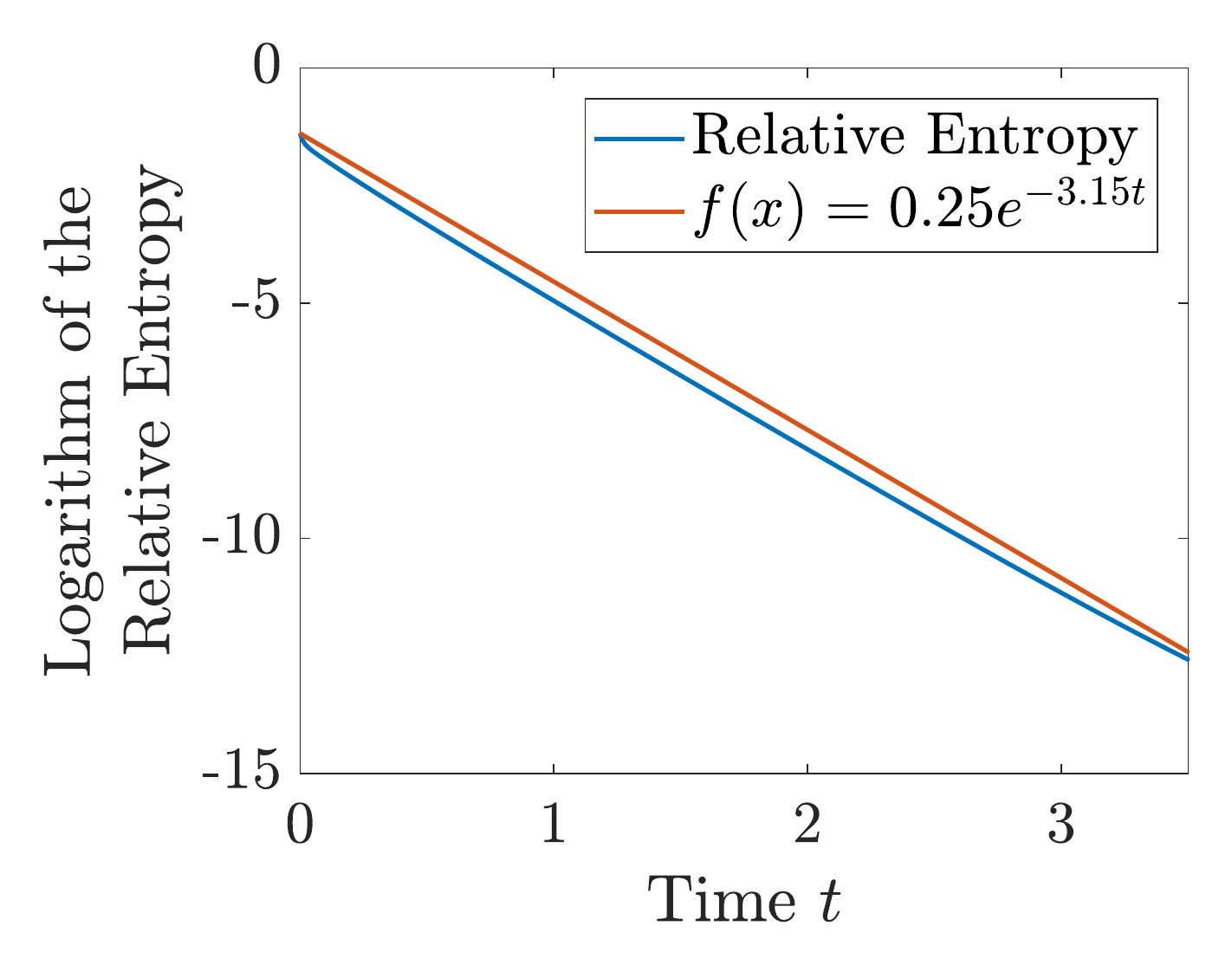} 
	\end{center}
\caption{\textbf{Relative Entropy:} The logarithm of the relative entropy for the one dimensional version of \eqref{Fokker-Planck1C} for $\alpha =1, \beta=0.9$ and $\rho_0(x) = -(x-0.5)^2+1$.}
	\label{1CPlotRelativeEntropyfunction}
\end{figure}

\vspace{2ex}
\noindent \textbf{Acknowledgments:} The authors acknowledge support by EXC 1003 Cells in Motion Cluster of Excellence, M\"unster, funded by the German science foundation DFG. MB was further supported by ERC via Grant EU FP7 - ERC Consolidator Grant 615216 LifeInverse.
The authors would like to thank Andreas P\"uschel and Danila di Meo (WWU M\"unster) for details on the biological background.

% \printbibliography
\bibliographystyle{alpha}
\bibliography{bib/bibliography.bib}

\end{document}